\documentclass[a4paper,11pt]{article}
\usepackage{amsfonts}
\usepackage{amsmath}
\usepackage{amssymb}
\usepackage{graphicx}
\usepackage{color}
\def\mathscr{\textbf}
\hoffset= - 0.9 cm

\def\cal{\mathcal}

\def\qed{\hfill \hbox{\hskip 6pt\vrule
width6pt height6pt depth1pt  \hskip1pt}
\smallskip}

\providecommand{\U}[1]{\protect\rule{.1in}{.1in}}
\newcommand{\R}{\mathbb R}
\newcommand{\N}{\mathbb N}

\newtheorem{theorem}{Theorem}

\newtheorem{corollary}[theorem]{Corollary}

\newtheorem{definition}{Definition}
\newtheorem{example}[theorem]{Example}

\newtheorem{lemma}[theorem]{Lemma}

\newtheorem{hypothesis}[theorem]{Hypothesis}
\newtheorem{proposition}[theorem]{Proposition}
\newtheorem{remark}[theorem]{Remark}

\newenvironment{proof}[1][Proof]{\noindent\textbf{#1.} }{\ \rule{0.5em}{0.5em}}

\newcommand{\bth}{\begin{theorem}}
\def\eth{\end{theorem}}
\newcommand{\bpr}{\begin{proposition}}
\newcommand{\epr}{\end{proposition}}
\newcommand{\bco}{\begin{corollary}}
\newcommand{\eco}{\end{corollary}}
\newcommand{\ble}{\begin{lemma}}
\newcommand{\ele}{\end{lemma}}
\newcommand{\bpf}{\begin{proof}}
\newcommand{\epf}{\end{proof}}
\newcommand{\bex}{\begin{example}}
\newcommand{\eex}{\end{example}}
\newcommand{\bdf}{\begin{definition}}
\newcommand{\edf}{\end{definition}}
\newcommand{\bre}{\begin{remark}}
\newcommand{\ere}{\end{remark}}

\newcommand{\beq}{\begin{equation}}
\newcommand{\eeq}{\end{equation}}
\newcommand{\bal}{\begin{aligned}}
\newcommand{\eal}{\end{aligned}}
\newcommand{\ben}{\begin{enumerate}}
\newcommand{\een}{\end{enumerate}}
\newcommand{\beqr}{\begin{eqnarray*}}
\newcommand{\eeqr}{\end{eqnarray*}}

\def\R{{\mathbb R}}

\def\L{\mathcal L}
\def\Re{\mathcal R}

\def\lan{{\langle}}
\def\ran{{\rangle}}
\def\hh{{\vskip 1mm \noindent }}
\def\vv{{\vskip 1mm}}
\def\Tr{\text Tr}

\begin{document}

\title{On weak uniqueness for  some degenerate SDEs
by  global $L^p$ estimates}


\author{ E. Priola
 \\ Department of Mathematics
\\ University of Torino \\
 {\small via Carlo Alberto 10, Torino} \\
 {\small enrico.priola@unito.it} } \maketitle

\begin{abstract}
 We prove
 uniqueness in law
  for
 possibly degenerate SDEs
having  a
linear part in the drift term.
Diffusion coefficients corresponding to   non-degenerate directions of the noise are  assumed to be   continuous.
 When the diffusion part is constant we recover
 the
  classical degenerate Ornstein-Uhlenbeck process
 which
 only
 has to
 satisfy the H\"ormander hypoellipticity condition. In the proof we  also use global  $L^p$-estimates for hypoelliptic Ornstein-Uhlenbeck operators
 recently proved  in Bramanti-Cupini-Lanconelli-Priola (Math. Z. 266 (2010))
  and
 adapt the
 localization procedure
 introduced by Stroock and Varadhan.
 Appendix contains a quite general localization principle for
  martingale problems.
  \end{abstract}


{\vskip 7 mm }

\vspace{2.5 mm}

\vspace{2.5 mm}
\section{Introduction}


In this paper we prove existence and weak uniqueness (or uniqueness in law) for
 possibly degenerate
 SDEs like
 \begin{equation} \label{SDE1}
 dZ_{t}= A Z_t dt + b(Z_t)dt + B(Z_t)dW_t,\qquad
   t \ge 0, \; Z_0 =z_0 \in \R^d,
\end{equation}
 where    $A$ is a $d \times d$ real matrix,
 $W=(W_t)$ is a standard  $r$-dimensional
  Wiener process, $r \ge 1$,
 $ B(z)= \begin{pmatrix}
 B_0(z) \\
0
\end{pmatrix},$ with  $B_0(z) \in \R^{d_0} \otimes \R^r$ (i.e., $B_0(z)$
is a real $d_0 \times r$-matrix, for any $z \in \R^d$), $1 \le d_0 \le d$, and
   $B(z) \in \R^d \otimes \R^r$, $z \in \R^d$. Moreover, we suppose that
 $$
 b(z) = \begin{pmatrix}
 b_0(z) \\
0
\end{pmatrix},
$$
 where $b_0 : \R^d \to \R^{d_0}$ ($\R^{d_0} \simeq \R^{d_0} \otimes \R$) is a Borel and locally bounded function.

Writing $z \in \R^d$ in the form
$z =
    \begin{pmatrix}
 x \\
y
\end{pmatrix}
  \simeq (x,y) \in \R^d$, with $x \in \R^{d_0}$ and $y \in \R^{d_1}$ (if $d_1 = d - d_0
  =0$ then $z=x$)
   and, similarly, $Z_t = (X_t, Y_t)$, we may rewrite \eqref{SDE1} as
\begin{equation} \label{SDE}
\begin{pmatrix}
 dX_t \\
dY_t
\end{pmatrix}
 = A
\begin{pmatrix}
 X_t \\
Y_t
\end{pmatrix}
dt \, + \,
\begin{pmatrix}
 b_0(X_t, Y_t) \\
0
\end{pmatrix} dt
  \, + \, \begin{pmatrix}
 B_0(X_t, Y_t) \\
 0
\end{pmatrix}dW_t,
 \end{equation}
 $ t \ge 0, \; (X_0, Y_0)= z_0 = (x_0, y_0)  \in \R^d.$
We assume that   $B_0 $ is continuous from $\R^d$ into $\R^{d_0} \otimes \R^r$ and also that the $d_0 \times d_0$ symmetric matrix $Q_0(z) = B_0(z) B_0(z)^*$ (here $B_0(z)^*$ denotes the adjoint matrix of $B_0(z)$) is  positive  definite for any $z \in \R^d$ (see  (i) and (iii) in Hypothesis \ref{hy}).
 Moreover,  for any $z_0 \in \R^d$, the Ornstein-Uhlenbeck process
 $dZ_t = AZ_t dt + B(z_0) dW_t$ must satisfy a hypoellipticity
  type condition (see (ii) in Hypothesis \ref{hy}).
 These assumptions allow to prove the main theorem which is about weak uniqueness (see Theorem \ref{main1}). The proof is based  on  a suitable version of Calderon-Zygmund $L^p$-estimates and on a variant of the localization principle  introduced by Stroock and Varadhan.

  We also  prove well-posedness of \eqref{SDE}, assuming in addition that  there exists   a  smooth Lyapunov function $\phi : \R^d \to \R_+$ (see Hypothesis \ref{hy1} and Theorem \ref{main2}). This function
 controls the growth of the  coefficients (cf. Chapter 10 in \cite{SV79}) and  gives  a sufficient condition for the existence of global solutions.
 In the standard case of $\phi(z)
 = 1+ |z|^2 = 1 + |x|^2 + |y|^2$ ($|\cdot|$ denotes the euclidean norm) we assume  that
  there exists $C>0$ such that
\begin{align}\label{diss}
    \text{Tr}(Q_0(x,y)) + 2 \langle A (x,y) , (x,y)\rangle +
   2  \langle b_0 (x,y) , x\rangle_{\R^{d_0}} \le C (1+ |z|^2),\; z=(x,y) \in \R^d
\end{align}
(here Tr denotes the trace and  $\langle \cdot , \cdot \rangle$  the inner product).

  Solutions to equation \eqref{SDE} appear as a natural generalization of OU processes.
 On the other hand  degenerate Kolmogorov operators $\L$
  associated to \eqref{SDE}
  (see \eqref{ll}) arise in
  Kinetic Theory (see \cite{DV} and the references
  therein) and in  Mathematical Finance (see
   the survey paper \cite{LPP}). In addition
 diffusion processes like $(Z_t)$
  appear   in stochastic motion of particles according to the Newton law
 (see,  for instance, \cite{Fr}).

If $d = d_0$, i.e., we are in  the case of a non-degenerate diffusion, weak uniqueness
(or uniqueness in law)  has been  proved in \cite{SV69} even in the case of time dependent coefficients (see \cite{Kr} for a different proof of uniqueness when the coefficients  are  independent of time).
 This has been done  by introducing  the important  localization principle. It states that
 uniqueness is a local result in that it
suffices to show that each starting point has a
 neighbourhood on which the coefficients of our SDE equal other coefficients for which uniqueness
 holds (cf. Theorem 6.6.1 in \cite{SV79}).
This principle combined with   global   $L^p$-estimates
 for  heat equations
 has been used  in \cite{SV69} to prove the
 uniqueness result.

The results in \cite{SV69} have been generalized in several papers about non-degenerate diffusions (see \cite{BPa,Kr1} and the references therein) by allowing  some discontinuous coefficients $B_0(z)$ (see \cite{Na} for a counterexample to uniqueness
with $d \ge 3$ and $B_0(z)$  measurable).

Weak uniqueness results are also available for some degenerate SDEs with non  locally Lipschitz coefficients (see  \cite{ABBP,  BP1, BP2, Br, Fi, LL, M}). Such results do not cover equations like \eqref{SDE} under our assumptions.
In particular  related degenerate SDEs with $d_0 < d$ are considered in \cite{Br, M}. In  \cite{M} (see also \cite{DM})
 the $d_0 \times d_0$ non-degenerate diffusion part
 has bounded H\"older continuous coefficients but it is not assumed that the drift term
   has  a  linear part like $AZ_t dt$ (in particular the second component of $b(z)$ in  \eqref{SDE1} can be different from 0).
   In \cite{Br} degenerate SDEs with time-dependent coefficients which  grow at most linearly  are considered;  these equations  have a  linear   part in the drift which has to satisfy a  lower-diagonal block form.

It seems to be  a hard problem to  prove weak uniqueness for SDEs as in \cite{M} with the non-degenerate diffusion part which is only continuous (a special result in this direction is Theorem 5.14 in \cite{Br}).

To establish our main result  (see Theorem \ref{main1}) we first prove in Section A.3 of appendix a variant of the localization principle of   Stroock and Varadhan (see, in particular, Theorem \ref{uni1} which is based on Theorem \ref{key}
  and Lemma \ref{stop}; these results
 provide  extensions
 of some related theorems in Chapter 4 of \cite{EK}).
 We cannot apply directly the localization principle as it is stated in
 Section 6.6 of  \cite{SV79}  since our SDE is degenerate  and we cannot localize our linear function $z \mapsto Az$ and then provide   the necessary analytic regularity results
 (cf. Remark \ref{stro}). In
  the proof of uniqueness we also use  global regularity results  for hypoelliptic Ornstein-Uhlenbeck operators $\L_0$ (see \eqref{loo1}) in $L^p$-spaces with respect to the Lebesgue measure recently proved in \cite{BCLP} (see, in particular, Theorem \ref{BCLP}).
 The regularity results in \cite{BCLP} are proved  using that
 $\L_0 - \partial_t$ is left invariant with respect to a suitable Lie group structure on $\R^{d+1}$ (see \cite{LP}); this group in general is not homogeneous.

\smallskip
The plan of the paper is as follows. In Section 2 we start with basic definitions and preliminary results about well-posedness of \eqref{SDE}. We also formulate our main results. In Section 3 we prove a uniqueness result for
\eqref{SDE} assuming additional hypotheses  on the coefficients (see Theorem \ref{cons}).  In that section we also establish some necessary analytic results for OU hypoelliptic operators $\L_0$.
 The complete uniqueness result is proved in Section 4 where we remove the additional hypotheses using the localization procedure.
Finally Appendix contains a quite general localization principle for martingale problems.

\paragraph{Basic assumptions.}
  Recall that $(e_i)_{i=1, \ldots, d}$ denotes the canonical basis on $\R^d$. Moreover, $\lan \cdot, \cdot \ran $ indicates the inner product in any $\R^n$, $n \ge 1$, and $|\cdot|$ denotes the  Euclidean norm in $\R^n$.
 \begin{hypothesis} \label{hy} {\em { (i)} \
     The  symmetric $d_0 \times d_0 $ matrix
  $Q_0(z) = B_0(z) B_0(z)^* $  is  positive definite, for any $z \in \R^d.$

\hh(ii)  There exists a  non-negative integer $k$,
 such
 that the vectors
\begin{align} \label{kal}
 \{ e_1 , \ldots, e_{ d_0}, A e_1 ,
 \ldots, Ae_{ d_0}, \ldots, A^k e_1 , \ldots, A^k
 e_{ d_0}\} \;\;\; \mbox{generate} \;\; \R^d;
 \end{align}
  we
 denote by $k$ the {\it smallest} non-negative integer such that \eqref{kal}
 holds (one has $0 \le k \le d-1$).

\hh (iii) $b_0 : \R^d \to \R^{d_0} $ is Borel and locally bounded;
$B_0: \R^d \to \R^{d_0} \otimes \R^r $ is continuous. \qed}
\end{hypothesis}
\begin{hypothesis} \label{hy1} {\em There exists a smooth Lyapunov function $\phi$ for \eqref{SDE}, i.e., there exists  a $C^2$-function $\phi : \R^d \to (0, +\infty)$ such that $\phi \to + \infty$ as $|z|\to +\infty$ and
\begin{align}\label{lia}
 {\L} \phi(z) \le C \phi(z),\;\; z \in \R^d,
\end{align}
 for some $C>0$;   $ \L$ is the possibly degenerate Kolmogorov
  operator related to \eqref{SDE},
 \begin{equation}\label{ll}
 \L f(z) = \frac{1}{2} \Tr(Q_0(z) D^2_x f(z))
+ \lan A z , D f(z)\ran  \, + \,
\lan b_0(z) , D_x f(z)\ran,
\end{equation}
 $f \in C^{2}_K(\R^d), \; z \in \R^d,$
 where $D f(z) = (D_x f(z), D_yf(z)) \in \R^d$ indicates the gradient of $f$ in $z$
 and $D^2f(z)$ denotes the Hessian matrix of $f$ in $z$,
 $$D^2f(z) = \begin{pmatrix}
 D^2_x f(z) & D^2_{xy} f(z)\\
 D^2_{xy} f(z) & D^2_{y} f(z)
\end{pmatrix}\in \R^{d} \otimes \R^d. \qed
 $$
}
\end{hypothesis}

 Note that  $d_1 =0$ if and only if $k=0$. In this case    $d = d_0$ and we have a non-degenerate SDEs with $B(z)= B_0(z)$ for which weak uniqueness is already known (see \cite{SV79}).

 By the
   H\"ormander condition on commutators, \eqref{kal}
      is equivalent to the   hypoellipticity of the operator
   ${\cal L}_0 - \partial_t $
  in $(d+1) $  variables  $(t, z_1, \ldots,  z_d)$;
     here
    ${\cal L}_0$ is the OU operator
    \begin{equation} \label{loo1}
   {\cal L}_0 u (z) = \frac{1}{2} \sum_{i,j =1}^{{ d_0}} q_{ij}
\partial_{x_i x_j}^2
   u (z) +
  \sum_{i,j =1}^d a_{ij} \, z_{j} \partial_{z_i } u (z), \;\;\;
  z \in \R^d,
 \end{equation}
where $Q_0 = (q_{ij})_{i,j =1, \ldots, d_0}$ is symmetric and positive definite on $\R^{d_0}$ and the  $a_{ij}$ are the components of the
 $d \times d$-matrix $A$; further $\partial_{x_i}$ and
 $ \partial_{x_i x_j}^2$ denote  partial derivatives.

 It is also well-known (see Section 1.3 in \cite{Za}) that
\eqref{kal} is equivalent to the fact that the symmetric $d \times d$ matrix
\begin{equation}\label{qtt2}
Q_t = \int_0^t e^{sA} Q e^{sA^*} ds \;\;\text{is positive definite for all} \; t>0,
 \;\; \text{with} \;
 Q  = \begin{pmatrix}
 Q_0 & 0\\
 0 & 0
\end{pmatrix};
\end{equation}
 here $e^{sA}$ denotes the exponential matrix of $sA$.

\hh \textbf{An example.}
Let us consider the following three-dimensional  example
\begin{equation}\label{ex1}
\begin{cases}
dx_t =  (- x_t^3 + \frac{y_t}{|y_t|})\, dt \,  + \,
a(x_t, y_t, z_t)
 \, dW_t
\\
dy_t = (x_t + y_t) dt
\\
dz_t = (y_t + z_t) dt,
\end{cases}
\end{equation}
where $(x_t, y_t, z_t ) \in \R^3$, $(x_0, y_0, z_0 )= \xi$. Here $W = (W_t)$ is a one-dimensional Wiener process. Thus  $d_0 =1$, $b_0(x,y,z) = -x^3 +\frac{y}{|y|}$, $A = \left(
\begin{array}{ccc}
    0 & 0 & 0 \\
  1 & 1 & 0 \\
  0 & 1 & 1
\end{array}
\right)$ and we can assume that $a$ is continuous and bounded and that $a^2$ is  positive  on $\R^3$. Note that here $k=2.$
 The associated degenerate Kolmogorov operator is $$
 {\cal L} =\text{\small{$\frac{1}{2}$}}a^2(x,y,z)\, \partial_{xx}^2 - x^3 \partial_x + \frac{y}{|y|} \partial_x + (x+y) \partial_y + (y+ z) \partial_z
$$
 and as  Lyapunov function we may consider  $\phi (x,y,z) = x^2 + y^2 + z^2 +1.$
 Hence  Hypotheses \ref{hy} and \ref{hy1}  hold and we can prove well-posedness for \eqref{ex1} or, equivalently, well-posedness  of the martingale problem for $\cal L$ starting from any initial distribution on $\R^3$.

\paragraph{Notations.}  We will use the letter $c$ or $C$ with subscripts for
finite positive
 constants whose  precise  value is unimportant.

 For a matrix $B \in \R^r \otimes \R^d$, $r \ge 1$, $d \ge 1$,  $\| B \|$
denotes its Hilbert-Schmidt norm.

The space $B_b(\R^d)$ denotes the Banach space of all real bounded and Borel functions $f: \R^d \to \R$ endowed with the supremum norm $\| \cdot \|_{\infty}$; its subspace
of all continuous functions is indicated by $C_b(\R^d)$. Moreover
  $C^{2}_K=C^{2}_K (\R^d) \subset C_b(\R^d)$ is the space of functions of class $C^2$ with compact support and similarly  $C^{\infty}_K (\R^d) \subset C_b(\R^d)$ is the space of functions of class $C^{\infty}$ with compact support. In addition we consider the space $C^{2}_b (\R^d) \subset C_b(\R^d)$ consisting of  all functions of class $C^2$ having   first and second  partial derivatives
which are bounded on $\R^d$.

 We also consider standard $L^p$-spaces $L^p(\R^d)$ with respect to the Lebesgue measure and indicate by $\| \cdot\|_p$ (or $\| \cdot\|_{L^p}$) the usual $L^p$-norm, $p \ge 1$. For measurable matrix-valued functions  $u : \R^d \to \R^r \otimes \R^d$ we also consider $\| u\|_p =
(\int_{\R^d} \| u(z)\|^p dz )^{1/p}$.

Finally by ${\cal P}(\R^d)$ we denote the set of all Borel probability measures on $\R^d$. A  probability space will be indicated with $(\Omega, {\cal F}, P)$ and $E$ (or $E^P$) will denote expectation with respect to $P$.


\section {Basic definitions and  main results }

Our definitions  will mainly follow  Chapter 4 in \cite{IW} (see also \cite{EK, SV79}).
 Let us consider the SDE
\begin{equation} \label{SDE13}
 Z_t
 =  Z_0 + \int_0^t b(Z_s) ds
   \, + \,
 \int_0^t B(Z_s)   d W_s,\, \;\; t \ge 0.
 \end{equation}
where $b: \R^d \to \R^d$ and $B : \R^d \to \R^d \otimes \R^r$ are Borel and locally bounded functions and $W = (W_t)$ denotes a $r$-dimensional Wiener process.  Note that equation \eqref{SDE1} is a special case of \eqref{SDE13}.

The corresponding Kolmogorov operator (generator) is
\begin{equation} \label{lll}
   {\tilde \L} f(z)=  \frac{1}{2} \Tr(B(z)B^*(z) D^2 f(z))
  \, + \, \lan b(z) , D f(z)\ran, \;\;\;   f \in C_K^{2}(\R^d), \; z \in \R^d.
\end{equation}
 Let $\mu \in {\cal P}(\R^d)$. Let us recall two related  notions of solutions.
\begin{definition} {\em A \textit {weak  solution} $Z = (Z_t)= (Z_t)_{t \ge 0}$ to \eqref{SDE13} with initial condition  $\mu$ is a continuous $d$-dimensional process (i.e., it has  continuous paths with values in $\R^d$)
defined on a probability space $(\Omega, {\cal F}, P) $ endowed with a reference filtration $({\cal F}_t)$ such that

\hh (i) there exists an $r$-dimensional $\cal F_t$-Wiener process $W = (W_t)$;

\hh (ii) $Z$ is ${\cal F}_t$-adapted and the law of $Z_0$ is $\mu$;

\hh (iii) $Z$ solves \eqref{SDE13} $P$-a.s.. }
\end{definition}
  \begin{definition}\label{dft}{\em A \textit{solution of the martingale problem}  for $(\tilde {\cal L}, \mu)$  is a continuous   $d$-dimensional process $Z= (Z_t)$ defined on some probability space $(\Omega, {\cal F}, P) $ such that,
 for any $f \in C_K^{2}(\R^d)$,
\begin{align}\label{mart4}
    M_t(f) = f(Z_t) - \int_0^t {\tilde \L} f(Z_s) ds, \;\; t \ge 0, \;\; \text{is a
    martingale}
\end{align}
(with respect to the natural filtration $({\cal F}_t^Z)$,
 where ${\cal F}_t^Z = \sigma(Z_s \, : \, 0 \le s \le t)$, i.e., ${\cal F}_t^Z$
  is the  $\sigma$-algebra generated by the random variables $Z_s $,
  $ 0 \le s \le t$),
  and moreover, the law of $Z_0$ is $\mu$. }
\end{definition}
Note that ${\tilde \L} : D({\tilde \L}) = C_K^{2}(\R^d) \subset C_b(\R^d)
 \to B_b (\R^d)$ satisfies Hypothesis \ref{bpt} in Appendix. This fact  is quite standard; we  sketch the proof in Remark \ref{fini1}.

\smallskip
If $Z$ is a weak solution on $(\Omega, {\cal  F}, P)$  an  application of It\^o's formula shows that $Z$ is also a martingale solution for $(\tilde \L, \mu)$.

Conversely, if there exists  a martingale solution $Z$ for $({\tilde \L}, \mu) $ on $(\Omega, {\cal  F}, P)$ then
there exists a stochastic basis $(\hat \Omega, \hat {\cal  F}, (\hat {\cal  F}_t), \hat P)$ on which there exists an $r$-dimensional $\hat {\cal  F}_t$-Wiener process and a weak solution $Y = (Y_t)$ for \eqref{SDE13} such that the law of $Y$ coincides with the one of $Z$ (for more details see Section IV.2 in \cite{IW} or Section 5.3 in \cite{EK}). Thus we have (cf. Proposition IV.2.1 in \cite{IW})
\begin{theorem} \label{iki} The existence of a weak solution to \eqref{SDE13} with initial condition $\mu$ is equivalent to
the existence of  a martingale solution for $({\tilde \L}, \mu)$.
\end{theorem}

The following  result is essentially  due to Skorokhod (for a proof   one can argue as in the proofs of  Theorems IV.2.3 and IV.2.4 in \cite{IW};  see also Theorem 5.3.10 in \cite{EK}); recall that the Lyapunov function provides a sufficient condition for the non-explosion of solutions.

\begin{theorem} \label{sko}
 If the coefficients $b$ and $B$ are continuous functions on $\R^d$ and we assume  the existence of   a Lyapunov function $\phi$ as in \eqref{lia} (i.e., ${\tilde \L} \phi$ $\le C \phi$ on $\R^d$, $\phi : \R^d \to (0, +\infty)$ is a $C^2$-function and $\phi \to + \infty$ as $|z|\to +\infty$)
 then  there exists at least  one weak solution to \eqref{SDE13} for any initial condition $\mu \in {\cal P}(\R^d)$.
\end{theorem}
  If the drift $b$ is not continuous (as it happens in \eqref{SDE1} where $b(z) = A z + \begin{pmatrix}b_0(z)\\ 0 \end{pmatrix}$, $z \in \R^d$) to get existence of solution in general  one needs  additional non-degeneracy of the noise. For instance, if $B=0$ and $b$ is discontinuos  there are  many examples of  deterministic equations $Z_t = Z_0 + \int_0^t b(Z_s)ds$ for which there is no existence of  solutions.

\begin{definition} {\em
We say that \textit{ weak uniqueness or uniqueness in law holds for \eqref{SDE13} with initial condition $\mu \in {\cal P}(\R^d)$} if given two weak solutions $Z$ and $Z'$ (even defined on different stochastic bases) such that the law of $Z_0$ and $Z'_0$ is $\mu$ they have the same finite dimensional distributions. Similarly we say that  \textit{ uniqueness in law holds for the martingale problem for $({\tilde \L}, \mu)$} (cf. Section  A.1).}
\end{definition}
It is clear that uniqueness in law  for $({\tilde \L}, \mu)$ implies uniqueness in law for \eqref{SDE13}; also the converse holds (see Corollary 3.3.5 in \cite{EK}). Indeed we have
\begin{theorem} \label{eki}
Uniqueness in law for \eqref{SDE13} holds with initial condition $\mu$ if and only if uniqueness in law for the martingale problem for $({\tilde \L}, \mu)$ holds.
\end{theorem}
\begin{definition} {\em Finally,  we say that \textit{  the martingale problem for ${\tilde \L}$ is  well-posed} if, for any $\mu \in {\cal P}(\R^d)$,   there exists a martingale solution
 for $({\tilde \L}, \mu)$ and, moreover,  uniqueness in law  holds  for  the martingale problem for $({\tilde \L}, \mu)$. Similarly, we can define well-posedness for \eqref{SDE13}.
}
\end{definition}
 Let us come back to our SDE \eqref{SDE1} associated to $\L$
 given in \eqref{ll}. This is our main  result.
\begin{theorem} \label{main1}  Assume  Hypothesis \ref{hy} and suppose that for any $x \in
  \R^d$ there exists a martingale solution  for  $(\L, \delta_x)$.

 Then
  the martingale problem for $\L$
is well-posed.
  \end{theorem}
 Using Theorems \ref{sko} and \ref{main1} and Corollary \ref{exi2} we  obtain
\begin{theorem} \label{main2}  Assume Hypotheses \ref{hy} and \ref{hy1}.
  Then
  the martingale problem for $\L$
is well-posed.
  \end{theorem}

 The proofs of Theorems \ref{main1} and \ref{main2} are postponed 
 to Section 4. In Section 3
we will concentrate on proving
 that {\it the martingale problem for $\L_1$,
\begin{equation}\label{ll1}
\L_1 f(z) = \frac{1}{2} \Tr(Q_0(z) D^2_x f(z))
+ \lan A z , D f(z)\ran, \;\; \; f \in C^2_K, \; z \in \R^d,
\end{equation}
 is well-posed assuming \eqref{kal}, the continuity of  $B_0: \R^d \to \R^{d_0} \otimes \R^r $ and the additional conditions 
\begin{align}\label{extra}
    \eta |h|^2 \le \lan Q_0 (z)h, h \ran \le \frac{1}{\eta} |h|^2, \;\; h \in \R^{d_0},\;z \in \R^d, \;\; \text{\it for some} \;\; \eta >0,
\end{align}
and \eqref{sup1}
($\L_1$ is a special case of $\L$).}

We finish the section with a technical remark mentioned after Definition \ref{dft}.
\begin{remark} \label{fini1} {
\em There exists a countable
set $H_0 \subset C^{2}_K(\R^d)$ such
 that for any $f \in C^{2}_K(\R^d) $,
we can find  a sequence $(f_k) \subset H_0$ satisfying
\begin{equation}\label{vii}
 \lim_{k \to \infty} ( \| f - f_k \|_{\infty} +   \| \tilde \L f_k -  \tilde \L f \|_{\infty}) =0.
\end{equation}
To prove the assertion consider  the separable Banach space $V = C_0(\R^d)\subset C_b(\R^d)$ consisting of all continuous functions vanishing at infinity  (it is endowed with $\| \cdot\|_{\infty}$).

Then introduce  $\Lambda_n =\{ (f, Df, D^2f ) \}_{f \in C_K^2(B_n)}$, where $C_K^2(B_n) = \{ f \in C_K^2(\R^d)$ with support$(f)$ $ \subset B_n\}$;   $B_n = B(0,n)$ is the open ball of center 0 and radius $n \ge 1$.

Identifying $\R^d \otimes \R^d $ with $\R^{d^2}$ we see that each $\Lambda_n$ is contained in the product metric space $V^{1+ d + d^2}$ which is also separable.
It follows that $\Lambda_n $ is separable and so there exists a countable set $\Gamma_n \subset C_K^2(B_n)$ such that $\{ (f, Df, D^2f ) \}_{f \in \Gamma_n}$ is dense in $\Lambda_n$. For any $f \in C_K^2(B_n)$ we can find a sequence $(f_{k}^n)_{k \ge 1}
\subset \Gamma_n$ such that
\begin{equation}\label{d56}
   \| f - f_k^n \|_{\infty} + \| Df - Df_k^n \|_{\infty} + \| D^2f - D^2f_k^n  \|_{\infty} \to 0,\;\; \text{as} \; k \to \infty.
\end{equation}
Define $H_0 = \cup_{n \ge 1} \Gamma_n$.  If $g \in C^2_K(\R^d)$ then $g \in C_K^2(B_{n_0})$, for some $n_0 \ge 1$, and we can consider $(f_k^{n_0})
\subset C_K^2(B_{n_0})$ such that \eqref{d56} holds with $f$ and $f_k^n$ replaced by $g$  and $f_k^{n_0}$.
Then we obtain easily \eqref{vii} with $f$ and $f_k$ replaced by $g$  and $f_k^{n_0}$ (note that $\| \tilde \L f_k^{n_0} -  \tilde \L g \|_{\infty}$ $= \sup_{|z|\le n_0}|\tilde \L f_k^{n_0}(z) -  \tilde \L g(z) |$).

}
\end{remark}

\section{ The martingale problem for $\L_1 $  under  an additional hypothesis}

\begin{theorem} \label{cons} Let us consider $\L_1$
in \eqref{ll1}  assuming  Hypothesis \ref{hy}
and also \eqref{extra} for some $\eta >0$.
There exists a positive constant $\gamma = \gamma (A, d_0,  \eta,  d)$ such that if
\begin{equation}\label{sup1}
\sup_{z \in \R^d} \| Q_0(z) - \hat Q_0 \| < \gamma,
\end{equation}
for some positive definite  symmetric matrix
 $\hat Q_0 \in \R^{d_0} \otimes \R^{d_0}$
 such that
 ${\eta} |h|^2 \le \lan \hat Q_0 h,h \ran \le {\frac{1}{\eta}}
  |h|^2,$ $ h \in \R^{d_0}$,
 then
  the martingale problem for $\L_1$
is well-posed.
  \end{theorem}
   To prove the result we need some
   analytic regularity results for $\L_1$ when $Q_0(z)$ is constant.
\subsection{Analytic regularity results for  hypoelliptic OU operators }

 Let us consider the OU operator
 \begin{equation}\label{ou3}
 \L_0 f(z) = \frac{1}{2} \Tr(Q D^2 f(z))
+ \lan A z , D f(z)\ran =
 \frac{1}{2} \Tr(Q_0 D^2_x f(z))
+ \lan A z , D f(z)\ran, \;\; \; f \in C^{2}_K,
\end{equation}
 $z \in \R^d,$ where
$
Q  = \begin{pmatrix}
 Q_0 & 0\\
 0 & 0
\end{pmatrix},
  $ and $Q_0$ is a symmetric  positive definite $d_0 \times d_0$ matrix such that
 \begin{align} \label{qoo}
   {\eta} |h|^2 \le \lan  Q_0 h,h \ran \le \frac{1}{\eta} |h|^2,\;\;\; h \in \R^{d_0},
  \end{align}
  for some $\eta>0$.  The associated OU process starting at $z \in \R^d$  solves the SDE
 \begin{equation}\label{ou34}
 Z_t^z =  z + \int_0^t A Z_s^z ds
   \, + \,
 \int_0^t \sqrt{Q}  \,  d W_s,\, \;\; t \ge 0.
 \end{equation}
 The corresponding Markov semigroup is given by
\begin{equation}\label{ou341}
P_t f(z) = E[f(Z_t^z)] = \int_{\R^d} f(e^{tA}z + y) N(0,Q_t) dy,
\end{equation}
where $f \in B_b(\R^d)$, $z \in \R^d$ and $N(0,Q_t)$
is the Gaussian measure with mean 0 and covariance operator $Q_t$
\begin{equation}\label{qtt}
Q_t = \int_0^t e^{sA} Q e^{sA^*} ds, \;\; t \ge 0.
\end{equation}
{\it We assume  that $Q_t$  is positive definite, for any $t>0$} (cf. \eqref{qtt2}).

We will investigate regularity properties of  the resolvent
$R(\lambda, \L_0) $ which is defined by
\begin{equation}\label{reou}
  R(\lambda, \L_0)  f(z)  = \int_{0}^{+\infty} e^{- \lambda t} E [f(Z_t^z)] dt = \int_{0}^{+\infty} e^{- \lambda t} P_t f(z) dt,\;\; f \in C^{2}_K(\R^d),
\end{equation}
 $\lambda >0$, $z \in \R^d$.
 Our starting point
 is the following regularity  result  proved in \cite{BCLP} (a previous result for
  non-degenerate OU operators was established in \cite{MPRS}).

\begin{theorem} \label{BCLP}Let $p \in (1, \infty)$.
  Let us consider the hypoelliptic OU operator $\L_0$ (i.e., we are assuming \eqref{qoo} and \eqref{kal} or \eqref{qtt2}$)$.
There exists  $C= C(\eta, A, d_0, d,p)$ such that, for any $v \in C^{\infty}_K(\R^d)$,  we have
\begin{equation}\label{stim0}
    \| D^2_x   v \|_p \le C (\| \L_0 v\|_p + \| v\|_p).
\end{equation}
 \end{theorem}
The previous result allows to prove
\begin{theorem} \label{PP}
Let us consider the hypoelliptic OU operator $\L_0$.
 Let $p \in (1, \infty)$.
There exists $\lambda_0= \lambda_0(A,p, d)>0$
and $C= C(\eta, A, d_0, d,p)$ such that, for any $f \in C^{2}_K(\R^d)$,
$\lambda > \lambda_0$, we have
\begin{equation}\label{stim1}
    \| D^2_x R(\lambda, \L_0)  f \|_p \le C \| f\|_p.
\end{equation}
 \end{theorem}
Before proving the theorem we establish two lemmas of independent interest.

\begin{lemma} \label{nu0} Let us consider the OU resolvent
 given in \eqref{reou} with $Q$ as in \eqref{qoo} and $A$ which
  satisfies \eqref{kal} . Let $f \in C^2_K(\R^d)$.
  There exists $\hat p = \hat p (\eta, d, d_0 , A) \ge 1$ such that if
 $p > \hat p$ then
\begin{equation}\label{d6}
    \sup_{z \in \R^d} | R(\lambda, \L_0) f(z) | \le
    \sup_{z \in \R^d} \int_{0}^{+\infty} e^{- \lambda t} |P_t f(z)| dt
 \le    C \| f\|_{p},\;\;\; \lambda > 0,
\end{equation}
 with
  $C = C(p,\eta, d, d_0, A) >0$ independent of $f$.
\end{lemma}
\begin{proof} (i) By changing variable and using H\"older inequality we find,
for $p \ge 1$, $t>0$, $z \in \R^d$,
$$
|P_t f(z)| =  \Big| c_d \int_{\R^d} f(e^{tA}z + \sqrt{Q_t} \, y) e^{- \frac{|y|^2}{2}
}
dy \Big|
$$$$
\le c_p  \Big (  \int_{\R^d} |f(e^{tA}z + \sqrt{Q_t} \, y)|^p  dy \Big)^{1/p}
= \frac{c_p}{ (\text{det}(Q_t))^{1/2p} }
  \Big (  \int_{\R^d} |f(e^{tA}z + w)|^p  dw \Big)^{1/p}
$$
$$
 = \frac{c_p}{ (\text{det}(Q_t))^{1/2p} } \| f\|_p.
$$
with $c_p$ independent of $z$. Setting
$u_{\lambda} = R(\lambda, \L_0) f$ we find
$$
 \| u_{\lambda}\|_{\infty} \le \sup_{z \in \R^d}
 \int_{0}^{+\infty} e^{- \lambda t} |P_t f(z)| dt
\le c_p \| f\|_p
\int_{0}^{+\infty} e^{- \lambda t} \frac{1}{ (\text{det}(Q_t))^{1/2p} } dt.
$$
Now we need to estimate   $\text{det}(Q_t)$, for $t>0,$  with  a constant possibly depending on $\eta$ (see \eqref{qoo}). We have
\begin{equation} \label{f56}
 \lan Q_t h,h \ran =   \int_0^t \lan Q e^{sA^*}h, e^{sA^*}h \ran ds \ge \int_0^t \lan I_{\eta} e^{sA^*}h, e^{sA^*}h \ran ds = \lan Q_t^{\eta} h,h \ran, \;\; h \in \R^d,
\end{equation}
where $I_{\eta}
  = \begin{pmatrix}
 \eta I_{0} & 0\\
 0 & 0
\end{pmatrix},
  $ with $I_0$ the $d_0 \times d_0$-identity matrix, and
$$
Q_t^{\eta} = \int_0^t e^{sA} I_{\eta} \, e^{sA^*} ds.
$$
Condition (ii) in Hypothesis \ref{hy} is equivalent to the  controllability Kalman condition
$$
 \text{rank}[B, AB, \ldots, A^{k}B ] =d,
$$
with $B =  I_{\eta}$. This is also equivalent to the fact that $Q_t^{\eta}$ is positive definite for any $t>0$ (see, for instance, Chapter I.1 in \cite{Za}).

Now  we use a result in \cite{seidman} (see also Lemma 3.1 in \cite{Lu}).
According to formulae (1.4) and (2.6) in \cite{seidman} (in  \cite{seidman} $Q_t^{\eta}$ is denoted by $W_t$) we have
$$
\| (Q_t^{\eta})^{-1}\| \sim \frac{c_1}{t^{2k +1}} \;\;\; \text{as} \; t \to 0^+.
$$
It follows that $\lan Q_t^{\eta} h, h \ran \ge c \, t^{2k +1}$, $t \in (0,1)$, $ |h| =1$. Using \eqref{f56} we easily obtain
\begin{align}\label{l1}
  \text{det}(Q_t)   \ge C t^{2k +1},\;\; t \in (0,1),
\end{align}
where
  $C = C(\eta,A, d_0, d)$. On the other hand,
  $\text{det}(Q_t) \ge$ $ \text{det}(Q_1) \ge C $, $t \ge 1$. It follows that
  $$
  \| u_{\lambda}\|_{\infty} \le \sup_{z \in \R^d}
 \int_{0}^{+\infty} e^{- \lambda t} |P_t f(z)| dt
\le c_p \| f\|_p
\int_{0}^{+\infty} \frac{C' e^{- \lambda t}}{ (t^{2k +1} \wedge 1)^{1/2p} } dt,
  $$
 $C' = C'(p,\eta,A, d_0, d) $. By choosing $p$ large enough we get easily assertion \eqref{d6}.
\end{proof}

\begin{lemma} \label{nu1} Assume the same assumptions of Lemma \ref{nu0}
and let $f \in C^2_K(\R^d)$.
 Then, for any $p \ge 1$ there exists $\lambda_0 = \lambda_0 (p, d , A) >0 $, and
  $C = C(p, d, A) >0$
  such that
\begin{equation}\label{sti1}
    \| R(\lambda, \L_0) f \|_p \le
        \frac{C}{\lambda} \| f\|_{p},
\end{equation}
\begin{equation}\label{sti2}
    \| D R(\lambda, \L_0) f \|_p \le
        \frac{C}{\lambda}  \| Df\|_{p}\, ,\;\;\; \| D^2 R(\lambda, \L_0) f \|_p
         \le
        \frac{C}{\lambda}  \| D^2f\|_{p}\, , \;\;\;  \lambda > \lambda_0.
\end{equation}
 Moreover,
 for any $\lambda > \lambda_0$ the function $u_{\lambda}= R(\lambda, \L_0) f
 \in C^2_b (\R^d)$  is the unique bounded classical solution to
 \begin{align}\label{eq}
    \lambda u - \L_0 u =f
 \end{align}
 on $\R^d$. Finally, we have, for $\lambda > \lambda_0$,  with $C = (p, d , A)$,
 \begin{align}\label{se}
   \lambda \| u_{\lambda} \|_p + \| \L_0 u_{\lambda}\|_p \le C \| f\|_p.
 \end{align}
  \end{lemma}
\begin{proof} Set $g_t(z) = f(e^{tA} z)$, $t \ge 0$, $z \in \R^d$.
By changing variable   we find
$$
P_t f(z) =    \int_{\R^d} g_t( z + e^{-tA}
 y) N(0, Q_t)dy
=    \int_{\R^d} g_t( z + w) N(0,e^{-tA}Q_t e^{-tA^*})dw.
$$
By the Young inequality we get, for $p \ge 1$,
$$
\| P_t f\|_p \le \| g_t \|_p = e^{- \frac {t}{p} Tr(A)} \| f\|_p.
$$
Hence, by using the Jensen inequality, we have for $\lambda > - Tr(A)$
\begin{gather*}
    \| u_{\lambda}\|^p_p = \int_{\R^d} \Big| \frac{1}{\lambda}
 \int_{0}^{+\infty} {\lambda} {e^{- \lambda t}} P_tf(z)  dt
 \Big|^p dz
\\
 \le \frac{1}{\lambda^p} \int_{\R^d}
 dz\int_{0}^{+\infty} {\lambda}{e^{- \lambda t}} |P_t f(z)|^p  dt
 \le
\lambda^{1-p}
 \int_{0}^{+\infty} {e^{- \lambda t}}e^{-  {t} Tr(A)} dt \;  \| f\|_p^p
\\
 \le \frac{\lambda^{1-p}}{\lambda + Tr(A)} \| f\|_p^p
\end{gather*}
and so \eqref{sti1} follows easily.

Concerning \eqref{sti2} note that, for any $h \in \R^d$,
\begin{equation} \label{rt7}
\lan D u_{\lambda}(z), h \ran
= \int_{0}^{+\infty} e^{- \lambda t}  P_t ( \lan Df (\cdot), e^{tA}h \ran) (z)
  dt.
\end{equation}
Indeed we have the following straightforward formulae
$$
\begin{array}{c}
\lan DP_t f(z), h \ran = P_t ( \lan Df (\cdot), e^{tA}h \ran) (z),
\\
 \lan D^2P_t f(z) [h], k \ran = P_t ( \lan D^2 f (\cdot) [e^{tA}h], e^{tA}k \ran) (z), \;\;\;h,k \in\R^d,\; t \ge 0,
\end{array}
$$
$z \in \R^d$. Starting from \eqref{rt7}
 the first estimate in \eqref{sti2}
can be proved arguing as in the proof of \eqref{sti1}. In a similar
way we get also the second estimate in \eqref{sti2}.

\hh Let us prove the final assertion.
 It is easy to see that there exists $\lambda_0=
 \lambda_0(A,d) >0$ such that
for $\lambda > \lambda_0$ we have that $u_{\lambda} \in C^2_b (\R^d)$.    Moreover, for any
$z \in \R^d$, differentiating under the integral sign we get
$$
 \L_0 u_{\lambda}(z) = \int_{0}^{+\infty} e^{- \lambda t}  \L_0 (P_t
 f) (z)
  dt
$$
$$
= \int_{0}^{+\infty} e^{- \lambda t}  \frac{d}{dt} (P_t
 f) (z) dt = - f(z) + \lambda u_{\lambda}(z),
 $$
so that $u_{\lambda}$ is a classical solution to $\lambda u_{\lambda}  - \L_0 u_{\lambda} =f$ ($u_{\lambda}$ is the unique bounded classical solution by the maximum principle).
 Finally, writing
$$
 \L_0 u_{\lambda} = - f + \lambda u_{\lambda}
$$
and using \eqref{sti1} we obtain \eqref{se}.
\end{proof}

\smallskip
\noindent \begin{proof}[Proof of Theorem \ref{PP}] The proof is divided into two steps.

\hh \textit{Step 1.} {\it We show that \eqref{stim0} holds even if
$v \in C^2_K(\R^d)$.}

\smallskip To this purpose take any $v \in C^2_K(\R^d)$
and consider standard mollifiers $(\rho_n) \subset C^{\infty}_K(\R^d)$ (i.e., $0 \le \rho_n \le 1$, $\rho_n(z)=0$ if
$|z|> \frac{2}{n}$, $\int \rho_n =1$, $\rho_n (z) = \rho_n(-z)$).
Define $v_n = v * \rho_n \in C^{\infty}_K (\R^d)$. According to \eqref{stim0} we have
\begin{equation} \label{cf4}
  \| D^2_x   v_n \|_p \le C (\| \L_0 v_n\|_p + \| v_n\|_p).
\end{equation}
It is not difficult to show that $\L_0 v_n \to \L_0 v$ in $L^p(\R^d)$ as $n \to \infty$, $p \ge 1$. We only show that $
\lan Az , D v_n(z)\ran$ $\to \lan Az , D v(z)\ran$ in $L^p(\R^d)$ as $n \to \infty$ (similarly, one can check that $\frac{1}{2}\text{Tr}(Q_0 D^2_x v_n)$ $\to \frac{1}{2}\text{Tr}(Q_0 D^2_x v)$ in $L^p(\R^d)$). We have
$$
\begin{array}{c}
\lan Az , D v_n(z)\ran = g_n(z) + h_n(z),
\\
g_n(z) = \int_{\R^d} \lan Az - Aw  , D v(w)\ran \, \rho_n(z- w) dw,
\\
h_n(z) =  \int_{\R^d} \lan Aw  , D v(w)\ran \, \rho_n(z- w) dw.
\end{array}
$$
By standard properties of mollifiers, $h_n \to \lan Az , D v(z)\ran$ in $L^p(\R^d)$ as $n \to \infty$. Concerning $g_n$, we find
$$
\int_{\R^d} |g_n(z)|^p dz \le \int_{\R^d} dz \int_{\R^d} |\lan Aw  , D v(z - w)\ran |^p \, \rho_n(w) dw
\le \frac{2^p\, \| A\|^p}{n^p} \, \| Dv\|_p^p
$$
which tends to 0 as $n \to \infty$.
 Since $\L_0 v_n \to \L_0 v$ in $L^p(\R^d)$, we can pass to the limit in \eqref{cf4} as $n \to \infty $ and get, for $p>1$,
$$
  \| D^2_x   v \|_p \le C (\| \L_0 v\|_p + \| v\|_p).
$$
\textit{Step 2.} {\it We consider $\lambda_0$ from Lemma \ref{nu1} and prove that
 $u = u_{\lambda} = R(\lambda , \L_0) f$
 verifies \eqref{stim1} for $\lambda > \lambda_0$}.

\hh From Lemma \ref{nu1} we already know several regularity properties of $u$. We will
 use these properties in the sequel.

 Let $\phi \in C^{\infty}_K(\R^d)$ be such that $0 \le \phi \le 1$ and $\phi (z) = 1$,
  $|z| \le 1$. Define $w_n(z) = u(z) \cdot \psi_n(z)$, $z \in \R^d,$
  where $\psi_n(z) = \phi(\frac{z}{n}) $,
  for $n \ge 1$. It is clear that each $w_n \in C^2_K$. Applying the first step we
  have
$$
  \| D^2_x   w_n \|_p^p \le \hat C (\| \L_0 w_n\|_p^p + \| w_n\|_p^p)
$$
which becomes
(for $h, k \in \R^{d_0}$, $h \otimes k \in \R^{d_0} \otimes \R^{d_0}$, with  $h \otimes k [w] = h \lan k, w \ran$, $w \in \R^{d_0}$)
\begin{gather*}
\int_{\R^d}\Big \| \frac{1}{n^2} u(z)  D^2_x \phi(\frac{z}{n}) +  \frac{1}{n}
D_x u(z) \otimes D_x \phi(\frac{z}{n}) +
 \frac{1}{n}
 D_x \phi(\frac{z}{n}) \otimes
D_x u(z)  + D^2_x u(z)
\phi(\frac{z}{n}) \Big \|^p_p dz
\\
\le C'  \Big( \|\L_0 u \|_p^p  +  \sup_{z \in \R^d} |\lan A {z} ,
D \phi({z})
\ran |  \cdot  \| u\|_p^p +
\\
+ \frac{1}{n^2} \| D^2_x \phi \|_{\infty} \, \| u\|_p^p +
\frac{1}{n} \| D_x \phi \|_{\infty} \, \| D_x u\|_p^p
   +  \| u\|_p^p\Big),
\end{gather*}
 with $C' = C'(\eta, A, d_0, d,p)>0$.
Now by the Fatou lemma (using also \eqref{sti2} in Lemma
\ref{nu1}) as $n \to \infty$ we find
\begin{gather*}
 \|   D^2_x u     \|_p
\le C_1  \big( \|\L_0 u \|_p  +  \, \| u\|_p \big)
\le C_1   \big( \|\L_0 u - \lambda u\|_p + \lambda
\|  u\|_p +  \, \| u\|_p \big)
\end{gather*}
with $C_1$ independent of $\lambda$. Using \eqref{eq}  we get
 (recall that $u = u_{\lambda}$)
\begin{gather*}
 \|   D^2_x u_{\lambda}    \|_p
\le C_1   \big( \| f\|_p +
C\| f \|_p +  \, \frac{C}{\lambda_0} \| f\|_p \big),
\end{gather*}
 for $\lambda > \lambda_0$
 and this gives the assertion.
\end{proof}

\subsection{An estimate for the resolvent of a  martingale solution}

Next we generalize estimate \eqref{d6}
 to the case in which we have a martingale solution
  for the operator $\L_1$ given in \eqref{ll1}.

\begin{theorem} \label{qvar}
 Let us consider $\L_1$
  assuming  Hypothesis \ref{hy}
and also \eqref{extra} for some $\eta >0$.
  Consider
$\hat p$ from  Lemma \ref{nu0}.
  There exists a positive constant $\gamma = \gamma (A, d_0,  \eta,  d)$
  such that if $Q_0(z)$ in \eqref{ll1} verifies
\begin{equation}\label{sup13}
\sup_{z \in \R^d} \| Q_0(z) - \hat Q_0 \|_{} < \gamma,
\end{equation}
for some positive definite  matrix $\hat Q_0 \in \R^{d_0} \otimes \R^{d_0}$
such that $ {\eta} |h|^2 \le \lan \hat Q_0 h,h \ran \le {\frac{1}{\eta}}
 |h|^2,$ $ h \in \R^{d_0},$  then
 any  solution $Y =(Y_t)= (Y_t^z)$ to the martingale problem for $(\L_1, \delta_z)$ verifies,
 for any $f \in C^{2}_K (\R^d)$, $p > \hat p$, $\lambda > \tilde \lambda_0> 0,$
where $\tilde \lambda_0 = \tilde \lambda_0(A,p,d),$
\begin{equation}\label{}
     \Big|  \int_{0}^{+\infty}
     e^{- \lambda t} E [f(Y_t^z)] dt \Big| \le C \| f\|_{p},
\end{equation}
for some constant $C = C(p,\eta, d, d_0, A) >0$.
 \end{theorem}
\begin{proof}
The proof is inspired by the  one of Theorem IV.3.3 in \cite{IW}
(see also Chapter 7 in \cite{SV79}) and uses Theorem \ref{PP}, Lemmas \ref{nu0} and \ref{nu1}. Let  $\omega >0 $ and $M>0$ be such that
\begin{equation} \label{ohi}
\| e^{tA} \| \le M e^{\omega t},\;\; t \ge 0.
\end{equation}
The constant $\tilde \lambda_0$ will be $\lambda_0 \vee \omega$ where $\lambda_0(A,p,d)$
is  given in Theorem \ref{PP}.

Given a martingale solution $Y$ there exists a stochastic basis $( \Omega,  {\cal  F}, ( {\cal  F}_t),  P)$
on which there exists a $d_0$-dimensional $ {\cal  F}_t$-Wiener process
 $W= (W_t)$ and a  solution $Z = (Z_t)= (Z_t^z)$ to
\begin{equation}\label{conv}
 Z_t =  e^{tA} z +
 \int_0^t e^{(t-s) A}\sqrt{Q (Z_s)}\, d W_s,\;\;\; t \ge 0,\;\;\;
 Q(z)=   \begin{pmatrix}
 Q_0(z) & 0\\
 0 & 0
\end{pmatrix},
\end{equation}
such that the law of $Y$ coincides with the one of $Z$ (for more details see Section IV.2 in \cite{IW} or Section 5.3 in \cite{EK}).
In the sequel to simplify notation we write $Z_t$ instead of $Z^z_t.$
Thus it is enough to show that, for a fixed $\lambda > \tilde \lambda_0$ we have
\begin{equation}\label{lap1}
\Big| \int_{0}^{+\infty} e^{- \lambda t} E [f(Z_t)] dt \Big| \le C \| f\|_p \, ,\;\;\; f \in C^{2}_K.
\end{equation}
Let us define new adapted processes $X^m = (X^m_t)$, $m \ge 1,$
 $
  X^{m}_t= Z_{\frac{k}{2^{m}} \, \wedge \, m}
 $
 for $t\in [\frac{k}{2^{m}},\frac{k+1}{2^{m}}[$ and  $k=0, 1, \ldots $; moreover  consider
\begin{gather*} Z_{t}^{m}=
e^{tA} z +
 \int_0^t e^{(t-s) A}\sqrt{Q (X_s^m)}\, d W_s,\;\;\; t \ge 0.
\end{gather*}
Since, for any $T>0$, $\lim_{m \to \infty}E [\sup_{t \in [0,T]} |Z_{t}^{m} - Z_{t}|^2 ]=0 $,  it is easy to check that
$$
    \int_{0}^{+\infty} e^{- \lambda t} E [f(Z_t^m)] dt
    \to \int_{0}^{+\infty} e^{- \lambda t} E [f(Z_t)] dt
$$
as $m \to \infty$, for any  $f\in {C}_{K}^{2}\left(\mathbb{R}^{d}\right)$, $\lambda >0$. Therefore the assertion follows if we prove that
\begin{equation}\label{lap12}
\Big| \int_{0}^{+\infty} e^{- \lambda t} E [f(Z_t^m)] dt \Big| \le C \| f\|_p,\;\;\; f \in C^{2}_K, \; \lambda > \tilde \lambda_0,
\end{equation}
 with \textit{$C= C(p, \eta, d, d_0, A)$ independent of $m$.} This will be achieved into three   steps.

\hh \textit{Step 1.} {\it We show that, for any $m \ge 1$, \eqref{lap12} holds with $C$ possibly depending on $m$.}

\smallskip We fix $f \in C^{2}_K$, $m \ge 1$, $\lambda>0$ and consider
\begin{gather} \label{civuo}	 V_m(\lambda,z) f := {E}\left[\int_{0}^{\infty}e^{-\lambda t}	 f\left(Z_{t}^{m_{}}\right)dt\right]\\
\nonumber =
 \sum_{k=0}^{m_{}2^{m_{}}-1}{E}\left[
\int_{\frac{k}{2^{m_{}}}}^{\frac{k+1}{2^{m_{}}}}
	e^{-\lambda t}f\left(Z_{t}^{m_{}}\right)dt\right] +{E}\left[\int_{m_{}}^{\infty}e^{-\lambda t}f
	\left(Z_{t}^{m_{}}\right)dt\right].
\end{gather}
Let us fix $k \in \{ 0, \ldots, m_{}2^{m_{}}-1\}$ and define
  \begin{gather*}
J_k = {E}	\Big[
\int_{\frac{k}{2^{m_{}}}}^{\frac{k+1}{2^{m_{}}}}
	e^{-\lambda t}f\left(Z_{t}^{m_{}}\right)dt\Big] = {{E}}
\Big[\int_{\frac{k}{2^{m{}}}}^{\frac{k+1}{2^{m{}}}}
e^{-\lambda t}{{E}}\left[f\left(Z_{t}^{m{}}
\right) / {\mathcal{F}}_{\frac{k}{2^{m{}}}}
\right]dt\Big]
\end{gather*}
(we are using conditional expectation with respect to $\mathcal{F}_{\frac{k}{2^{m{}}}}$). If we set
$$
U = e^{ k/2^{m} A}z +
 \int_0^{k/2^m} e^{(k/2^{m}\, -s) A}\sqrt{Q (X_s^m)}\, d W_s
$$
then, by a well-known property of conditional expectation
 (using also that
 $$ \int_{k/2^m}^t e^{(t-r) A}\sqrt{Q (y_2)}\, d W_r $$ is independent of ${\mathcal{F}}_{\frac{k}{2^{m_{}}}}$ for any $y_2 \in \R^d$)
we have, for $t \in [\frac{k}{2^{m_{}}}, \frac{k+1}{2^{m_{}}}[$,
\begin{gather*}
{{E}}\left[f\left(Z_{t}^{m_{}}\right)|
{\mathcal{F}}_{\frac{k}{2^{m_{}}}}\right]
\\
 =
{{E}}\Big
[f\Big(e^{(t-k/2^{m})A}\, U +  \int_{k/2^m}^t e^{(t-s) A}\sqrt{Q (Z_{\frac{k}{2^{m}}})}\, d W_s  \Big)
/ {\mathcal{F}}_{\frac{k}{2^{m_{}}}}\Big]
= F_{}(t - \frac{k}{2^{m}}, U, Z_{\frac{k}{2^{m}}})
\end{gather*}
where
$$
F_{}\big(s,y_1, y_2 \big) ={{E}} \Big
[f\Big(e^{s A} y_1 +  \int_{0}^{s} e^{(s -r) A}\sqrt{Q (y_2)}\, d W_r  \Big)\Big]
$$
(note  that
$$
F_{}\big(t - \frac{k}{2^{m}},y_1, y_2 \big)=
{{E}} \Big
[f\Big(e^{(t-k/2^{m}) A}y_1 +  \int_{k/2^m}^t e^{(t-r) A}\sqrt{Q (y_2)}\, d W_r  \Big)\Big]).
$$
 It follows that (recall that  $U$ depends also on $k$)
\begin{gather*}
J_k =
{{E}}\left[\int_{\frac{k}{2^{m_{}}}}^{\frac{k+1}
{2^{m_{}}}}e^{-\lambda t}{{E}}\left[f\left(Z_{t}^{m_{}}\right) / {\mathcal{F}}_{\frac{k}{2^{m_{}}}}\right]dt\right]
=\int_{\frac{k}{2^{m_{}}}}^{\frac{k+1}
{2^{m_{}}}}e^{-\lambda t}{{E}}\left[
F_{}(t - \frac{k}{2^{m}}, U, Z_{\frac{k}{2^{m}}})
\right]dt
 \\
=
\int_{0}^{\frac{1}{2^{m_{}}}}
e^{-\lambda\left(s+\frac{k}{2^{m_{}}}\right)}
{{E}}
\left[ F_{}(s, U, Z_{\frac{k}{2^{m}}}) \right]
ds.
\end{gather*}
Therefore, for any $k =0, \ldots, m_{}2^{m_{}}-1,$
$$
|J_k| \le \int_{0}^{\frac{1}{2^{m_{}}}}
e^{-\lambda\left(s+\frac{k}{2^{m_{}}}\right)}
{{E}} [| F_{}(s, U, Z_{\frac{k}{2^{m}}})|]
ds \le
\int_{0}^{+\infty}
e^{-\lambda s }
{{E}} | F_{}(s, U, Z_{\frac{k}{2^{m}}})|
ds .
$$
Now it is crucial to observe that by Lemma
\ref{nu0} we have, for any $y_1, y_2 \in \R^d,$  $p> \hat p$, $\lambda > \lambda_0$,
\begin{gather} \label{be2}
 \int_{0}^{+\infty}
e^{-\lambda s }
 | F_{}(s, y_1, y_2)|
ds  \le C \| f\|_p,
\end{gather}
where $C = C(\eta, d, d_0, A,p)>0$ is
independent of $y_1$ and $y_2$.
 Indeed $ F_{}(t, y_1, y_2)$ coincides
  with the OU semigroup in \eqref{ou341} with
   $y_1 =z$ and $Q$ replaced by $Q(y_2)=   \begin{pmatrix}
 Q_0(y_2) & 0\\
 0 & 0
\end{pmatrix}
  $; note that   $Q_0(y_2)$ verifies \eqref{qoo}
   by \eqref{extra}.

It follows that $|J_k| \le  C \| f\|_p$, for any $k =0, \ldots, m_{}2^{m_{}}-1$. Similarly,
 using that
$$I= {E}	\left[
\int_{m}^{\infty}
	e^{-\lambda t}f\left(Z_{t}^{m_{}}\right)dt\right] = {{E}}
\left[\int_{m}^{\infty}
e^{-\lambda t}{{E}}\left[f\left(Z_{t}^{m{}}
\right) / {\mathcal{F}}_{m}
\right]dt\right],
$$
we find the estimate $|I| \le C \| f\|_p$.
 Returning to \eqref{civuo} we get
\begin{gather*}
\Big |{E}\left[\int_{0}^{\infty}e^{-\lambda t}	 f\left(Z_{t}^{m_{}}\right)dt\right] \Big|
\le \sum_{k=0}^{m_{}2^{m_{}}-1} |J_k| \,
	 +\, \Big |{E}\left[\int_{m_{}}^{\infty}e^{-\lambda t}f
	\left(Z_{t}^{m_{}}\right)dt\right]\Big |
\le m 2^{m} \, C \, \| f\|_p
\end{gather*}
which shows \eqref{lap12} with a constant
possibly depending on $m$.

\hh  \textit{Step 2.} {\it We establish  the following identity, for any $f \in C^2_b(\R^d)$, $\lambda>\omega$ (see \eqref{ohi}$)$,
\begin{gather} \label{tor1}
\lambda\int_{0}^{\infty}e^{-\lambda t}{E}\left[f\left(Z^m_{t}\right)\right]dt=
f\left(z\right)+ {E}\int_{0}^{\infty}e^{-\lambda t} \mathcal{L}_m
f\left(t,Z^m_{t}\right)dt,
\end{gather}
 with a suitable operator $\L_m$.}

Consider first $f \in C^{2}_K(\R^d)$ and fix $m \ge 1$. Writing It\^o's formula
for  $f(Z_t^m)$ and taking expectation we find
\begin{align*}
    E f(Z_t^m)= f(z) + E \int_0^t \langle A Z_s^m, Df(Z_s^m) \rangle ds
   \, + \,
 \frac{1}{2}\int_0^t E [\text{Tr}({Q (X_s^m)} D^2 f(Z_s^m))] ds,
\end{align*}
 $ t \ge 0.$ Defining the operator
 \begin{align*}
    \L_m f(s,z) = \frac{1}{2} \Tr(Q(X_s^m) D^2 f(z))
+ \lan A z , D f(z)\ran, \;\; \; f \in C^{2}_K(\R^d), \; z \in \R^d,
 \end{align*}
 with random coefficients, we see that
$
    E [f(Z_t^m)]= f(z) + E \int_0^t \L_m f(s,Z_s^m)ds.
$
Using the Fubini theorem we find
\begin{multline} \label{bau1}
\int_{0}^{\infty}e^{-\lambda t}{E}\left[\int_{0}^{t}
\mathcal{L}_m f \left(s,Z^m_{s}
\right)ds\right]dt
\\
={E}\left[\int_{0}^{\infty}\mathcal{L}_m f
\left(s, Z^m_{s}\right)ds \int_{s}^{\infty}e^{-\lambda t}dt\right]=\frac{1}{\lambda}{E}
\left[\int_{0}^{\infty}e^{-\lambda t}\mathcal{L}_m f\left(t,Z^m_{t}
\right)dt\right].
\end{multline}
It follows \eqref{tor1} for $f \in C^{2}_K(\R^d)$.
Now a simple approximation argument shows that \eqref{tor1}
holds even for $f \in C^2_b(\R^d)$.
To this purpose note  that  $E [ |Z_t^m|^2] \le  C(z,M, \omega, \eta) \, e^{2\omega t},$ $t \ge 0$ (this implies that $E [|\mathcal{L}_m f(t,Z^m_{t} )|] \le \tilde C \, e^{\omega t}$, $m \ge 1$, $t \ge 0$, where $\tilde C $ is independent of $t$).


\hh  \textit{Step 3.}
{\it We prove assertion \eqref{lap12}  with $C$ independent of $m$.}

\hh Using hypothesis \eqref{sup13} let
 $\hat \L_0$ be the hypoelliptic OU operator
   associated to $A$ and $\hat Q$ where
 $$
\hat Q  = \begin{pmatrix}
 \hat Q_0 & 0\\
 0 & 0
\end{pmatrix}.
$$
  We write
\begin{equation}\label{ba1}
     \L_m f(s,z) = \hat \L_0 f(z) + \Re_{m} f(s,z),
\end{equation}
$$
 \Re_{m} f(s,z) =   \frac{1}{2} \Tr([Q_0(X_s^m) -
  \hat Q_0] D^2_x f(z)),
  \;\; \; f \in C^2_b(\R^d), \; z \in \R^d, \; s \ge 0.
$$
Recall that
$$
V_m(\lambda,z)f =
\int_{0}^{\infty}e^{-\lambda t}{E}\left[f\left(Z^m_{t}\right)\right]dt,
\;\; f \in C^2_b(\R^d);
$$
we can rewrite \eqref{tor1}  as
\begin{gather} \label{to1}
\nonumber \lambda V_m(\lambda,z)f =
f\left(z\right)+ {E}\int_{0}^{\infty}e^{-\lambda t} \hat \L_0
f\left(Z^m_{t}\right)dt
- \lambda {E}\int_{0}^{\infty}e^{-\lambda t}
f\left(Z^m_{t}\right)dt
\\
+  \lambda {E}\int_{0}^{\infty}e^{-\lambda t}
f\left(Z^m_{t}\right)dt
+ {E}\int_{0}^{\infty}e^{-\lambda t} \Re_{m}
f\left(t,Z^m_{t}\right)dt.
\end{gather}
By taking
$$f = R(\lambda,\hat \L_0 )g = R(\lambda)g,
$$
for $g \in C^{2}_K(\R^d)$ ($R(\lambda,\hat \L_0 )g$  is defined
as in \eqref{reou} with $\L_0$ replaced by $\hat \L_0$)  and using that
 $(\lambda - \hat \L_0) R(\lambda,\hat \L_0 )g =g$ (see \eqref{eq}),
 we obtain from the above identity
\begin{gather*}
\nonumber \lambda V_m(\lambda,z) [R(\lambda)g]  =
  R(\lambda)g\left(z\right) - V_m(\lambda,z) g \\
 \nonumber
   + \lambda V_m(\lambda,z) [R(\lambda)g]
   + {E}\int_{0}^{\infty}e^{-\lambda t} \Re_{m}
[R(\lambda)g]\left(t,Z^m_{t}\right)dt.
\end{gather*}
We find, for any $g \in C^{2}_K(\R^d)$,  $m \ge 1,$ $\lambda>\omega$, $z \in \R^d$,
\begin{align} \label{chi}
V_m(\lambda,z) g  = R(\lambda)g\left(z\right)  + {E}\int_{0}^{\infty}e^{-\lambda t} \Re_{m}
[R(\lambda)g]\left(t,Z^m_{t}\right)dt.
\end{align}
Now by the first step we know that for $p > \hat p$, $\lambda> \tilde \lambda_0$,
$z \in \R^d$, $m \ge 1$,
$$
\| V_m(\lambda,z)  \|_{L (L^p; \R)} = \sup_{g \in C^{2}_K,
 \, \|g\|_{L^p(\R^d)} \le 1} |V_m(\lambda,z) g| < + \infty.
$$
Using Lemma \ref{nu0} and condition \eqref{sup13}, we find that,  for $\lambda > \tilde \lambda_0$,
$$
|V_m(\lambda,z) g | \le  |R(\lambda)g\left(z\right)|
$$$$+ \frac{1}{2} {E}\int_{0}^{\infty}e^{-\lambda t} |\Tr\big([Q_0(X_s^m) -
  \hat Q_0] \, D^2_x
R(\lambda)g\left(Z^m_{t}\right) \big)|dt
$$
 $$
\le C \| g\|_p + \frac{\gamma}{2} {E}\int_{0}^{\infty}e^{-\lambda t} \| D^2_x
 R(\lambda)g\left(Z^m_{t}\right)\|_{} \, dt
\le C \| g\|_p + \frac{\gamma}{2} V_m(\lambda,z)
\| D^2_x
R(\lambda)g\|_{}
$$
(we are considering $V_m(\lambda,z)$ applied to the function $z \mapsto \| D^2_x
R(\lambda)g(z)\|_{}$)
with $C = C(d,d_0, \eta,A,p)$. By taking the supremum over $\Lambda_1 = \{ g \in C^{2}_K, \, \|g\|_{L^p(\R^d)} \le 1 \}$, we find
$$
\| V_m(\lambda,z)  \|_{L (L^p; \R)} \le C + \frac{\gamma}{2} \| V_m(\lambda,z)  \|_{L (L^p; \R)}
\, \cdot \,  \sup_{ g \in \Lambda_1 }\| D^2_x
[R(\lambda)g]\|_{L^p(\R^d)}.
$$
Now we use Theorem \ref{PP} to deduce that, for   any $\lambda >
\tilde \lambda_0$, we have
$$
\sup_{ g \in \Lambda_1 }\| D^2_x
[R(\lambda)g]\|_{L^p} \le C'
$$
with $C' = C'(d,d_0, \eta,A,p)$.
 By choosing $\gamma $ small enough ($\gamma < \frac{1}{C'}$) we get that
$$
\| V_m(\lambda,z)  \|_{L (L^p ; \R)} \le 2 C,\;\;\; \lambda > \tilde \lambda_0,
$$
with $C$ which is also independent of $m \ge 1$.
This proves \eqref{lap12} and finishes the proof.
\end{proof}

\subsection{Proof of Theorem \ref{cons} }

Existence of martingale solutions can be proved using Theorem \ref{sko}. Indeed
${\cal L}_1$ is a very special case of the operator $\tilde \L$ (see \eqref{lll}) for which the existence of martingale solutions follows if we prove the existence of a Lyapunov function $\phi$.
Since $Q_0$ appearing in ${\cal L}_1$ is a bounded function  we can consider $\phi(z) = 1 + |z|^2$  and get the existence of martingale solutions.

\smallskip
Let us concentrate on uniqueness of martingale solutions.

We will use Theorems \ref{PP} and \ref{qvar}. The constant $\gamma$
 appearing in
 \eqref{sup1} will be the same constant as in Theorem \ref{qvar}.

  According to Corollary \ref{ria} to prove that
  the martingale problem for $\L_1$
is well-posed
 it is enough to fix any $z \in \R^d$ and prove that
if  $X^1= (X^1_t)$ and $X^2=(X^2_t)$
are two   solutions for the martingale problem for $(\L_1, \delta_z)$ (defined, respectively, on
 $(\Omega_1, {\cal F}_1, P_1)$ and $(\Omega_2, {\cal F}_2, P_2)$)
then they have the same  one dimensional marginal distributions.

\smallskip
To this purpose we first consider $\hat p$ from Theorem \ref{qvar} and fix any $p>
\hat p$. Then we take
$\tilde \lambda_0 = \tilde \lambda_0(A, p, d) >0$ from  Theorem \ref{qvar} (recall that
$\tilde \lambda_0 \ge \lambda_0 $ where $\lambda_0(A,p,d)$
is  given in Theorem \ref{PP})  and
define
 \begin{align}\label{eq13}
G_i(\lambda,z)f =
\int_{0}^{\infty}e^{-\lambda t}{E_i}\left[f\left(X^i_t\right)\right]dt,
  \;\; i =1,2,\;\; f \in C^2_K(\R^d), \; \lambda > {\tilde \lambda_0}.
 \end{align}
If we prove that
  for $\lambda > {\tilde \lambda_0}$ we have
 \begin{align}\label{133}
G_1(\lambda,z)f =G_2(\lambda,z)f,
 \end{align}
 for $f \in C^2_K(\R^d)$, then   by a well-known property of the Laplace transform
  we get that   $E[f(X^1_t)] = E[f(X^2_t)] $, $t \ge 0$, $f \in
 C^2_K(\R^d)$ and this shows that $X^1$ and
 $X^2$ have the same   one dimensional marginal distributions.

   To check \eqref{133} we will  also use some arguments from the proof of Theorem \ref{qvar}.

Let us fix $i=1,2$. By the martingale property we deduce that
$$
E_i[f(X^i_t)] = f(z) + E_i\int_0^t \L_1 f(X^i_s)ds, \;\;\; f \in C^2_K,\; t \ge 0.
$$
Arguing as in the proof of \eqref{tor1} we obtain
$$
\lambda\int_{0}^{\infty}e^{-\lambda t}{E_i}\left[f\left(X^{i}_t\right)\right]dt=
f\left(z\right)+ {E_i}\int_{0}^{\infty}e^{-\lambda t} \mathcal{L}_1
f\left(X^{i}_t\right)dt
$$
or, equivalently,
\begin{align}\label{tor4}
\lambda G_i(\lambda,z)f =
f\left(z\right)+
G_i(\lambda,z) \mathcal{L}_1
f.
\end{align}
Note that \eqref{tor4} holds even for $f \in C^2_b(\R^d)$ (see the comment after
\eqref{bau1}).
 Using hypothesis \eqref{sup1} let
 $\hat \L_0$ be the OU operator
   associated to $A$ and $\hat Q$ where
 $$
\hat Q  = \begin{pmatrix}
 \hat Q_0 & 0\\
 0 & 0
\end{pmatrix}.
$$
  We write, similarly to \eqref{ba1},
$$
     \L_1 f(z) = \hat \L_0 f(z) + \Re  f(z),
$$
$$
 \Re f(z) =   \frac{1}{2} \Tr([Q_0(z) -
  \hat Q_0] D^2_x f(z)),
  \;\; \; f \in C^2_b(\R^d), \; z \in \R^d.
$$
We can rewrite \eqref{tor4}  as
\begin{gather} \label{to4}
\nonumber   G_i(\lambda,z) (\lambda f - \hat \L_0 f)  =
f\left(z\right)+  G_i(\lambda,z) \Re_{}
f,\;\;\; f \in C^2_b(\R^d).
\end{gather}
By taking $f = R(\lambda,\hat \L_0 )g = R(\lambda)g$, $g \in C^{2}_K(\R^d)$ ($R(\lambda,\hat \L_0 )g$  is defined as in \eqref{reou} with $\L_0$ replaced by $\hat \L_0$) we obtain from the above identity
\begin{align} \label{chi1}
G_i(\lambda,z) g  = R(\lambda)g\left(z\right)  +
G_i(\lambda,z) \Re_{}[R(\lambda)g],
\end{align}
$g \in C^{2}_K(\R^d)$,   $\lambda> {\tilde \lambda_0}$, $i=1,2$.
 Define  $T(\lambda,z) : C^2_K \to \R$,
$$
T(\lambda,z) g = G_1(\lambda,z) g  - G_2(\lambda,z) g.
$$
We have by \eqref{chi1}
\begin{align}\label{tl}
    T(\lambda,z) g = T(\lambda,z) (\Re_{}[R(\lambda)g]).
\end{align}
By using Theorem \ref{qvar} we know that $T(\lambda,z)$, for any $\lambda > {\tilde \lambda_0}$,
can be extended to a bounded linear operator from $L^p(\R^d)$ into $\R$.
 By \eqref{tl} we find, using also \eqref{sup1},
$$
\| T(\lambda,z)    \|_{L (L^p ; \R)}  =\sup_{g \in \Lambda_1} |T(\lambda,z)  g | \le
  \frac{\gamma}{2} \| T(\lambda,z)    \|_{L (L^p; \R)}
\, \cdot \,  \sup_{g \in \Lambda_1}\| D^2_x
[R(\lambda)g]\|_{L^p}.
$$
where $\Lambda_1 = \{ g \in C^{2}_K, \, \|g\|_{L^p(\R^d)} \le 1 \}$.
 Now by Theorem \ref{PP} we know that, for   any $\lambda >
{\tilde \lambda_0}$,
$$
\sup_{g \in \Lambda_1}\| D^2_x
[R(\lambda)g]\|_{L^p} \le C',
$$
with $C' = C'(d,d_0, \eta,A,p)$.
 By choosing $\gamma $ small enough ($\gamma = \frac{1}{C'}$) we get that
\begin{equation}\label{cf7}
\| T(\lambda,z)    \|_{L (L^p ; \R)} =0,\;\;  \lambda > {\tilde \lambda_0}.
\end{equation}
Note that it is important that  $C'$ is independent of $\lambda$ (at least for $\lambda$ large enough); otherwise we
 should choose for any $\lambda$  a suitable constant $\gamma= \gamma(\lambda)$  and we could
  not conclude the   argument.

Formula \eqref{cf7}  shows that  \eqref{133} holds  and this finishes the proof.
\qed

\section{ Proofs of Theorems \ref{main1} and \ref{main2}}

\vv \textbf{ Proof of Theorem \ref{main1}.}
We will    apply the localization principle  (see Theorem \ref{uni1}) with ${ {\mathcal A}} = \L$ and $D({{\mathcal A}}) = C^2_K(\R^d)$.
The proof is divided into three steps.
In the first step we construct a suitable covering of $\R^d$;  in the second step we define  suitable operators ${\mathcal A}_j$ according to Theorem \ref{uni1}. In the final step we prove well-posedness of the martingale problem associated to each ${\mathcal A}_j$.

\hh{\it I Step.} There exists
     a  countable set of points $(z_j) \subset \R^d$,
$j \ge 1$, and numbers $\delta_j >0$ and $\eta_j >0$, $\eta_{j +1} \le \eta_j$, with the following properties:

\smallskip
(i) the open balls $B(z_j,  \delta_j)$
  form a covering for $\R^d$;

(ii)   we have:
\begin{gather} \label{ett}
  \eta_j |h|^2 \le  \langle Q_0 (z) h,h \rangle \le \frac{1}{\eta_j} |h|^2,\;\;\; h \in
\R^{d_0},\;\; z \in B(z_j, 2 \delta_j),
\\
\nonumber
\| Q_0(z) - Q_0(z_j) \| < \gamma_j, \;\;\;  z \in B(z_j, 2 \delta_j),
\end{gather}
where  $\gamma_j = \gamma (A, d_0,  \eta_j,  d)$ and $\gamma $ is given in Theorem \ref{cons}.

\smallskip
In order to  construct the previous  covering,
let $C_r$ be the closed ball of center $0$ and radius $r>0$ and let  $B(z,  r)$ be the open ball of center $z \in \R^d$ and radius $r$. We have
$
\R^d = \cup_{k \ge 1} D_k,
$
where $D_1 = C_1$ and $D_k = C_{k} \setminus B(0,k-1)$, $k \ge 2$, are compact sets. Let us consider $D_1$. There exists $\eta_0 >0$ such that
 \begin{equation} \label{ba34}
  \eta_0 |h|^2 \le  \langle Q_0 (z) h,h \rangle \le \frac{1}{\eta_0} |h|^2,\;\;\; h \in
\R^{d_0},\;\; z \in D_{1}.
\end{equation}
Indeed, since $z \mapsto Q_0(z)$ is continuous,  the minimum and the maximum eigenvalue of $Q_0(z)$ depend  continuously on $z$ (recall that $Q_0(z)$ is a positive definite $d_0 \times d_0$ matrix, for any $z \in \R^d$).
 In a similar way, there exists $\eta_k >0$ such that
\begin{equation*}
  \eta_{k-1} |h|^2 \le  \langle Q_0 (z) h,h \rangle \le \frac{1}{\eta_{k-1}} |h|^2,\;\;\; h \in
\R^{d_0},\;\; z \in D_{k},\;\; k \ge 1.
\end{equation*}
We may assume that $\eta_k \le \eta_{k-1}$ (hence, for instance, \eqref{ba34} holds for $z \in D_1 \cup D_2$ when $\eta_0$ is replaced by $\eta_1$).
  Using  the uniform continuity of $Q_0(z)$ on $D_1 \cup D_2$ and a compactness argument, we can find a finite sequence $(u_n^{1}) \subset D_1$ and numbers $(r_n^1) \subset (0, 1/2)$
    such that  $D_1 \subset
\cup_{n} B(u_n^{1}, r_n^1) $  and
$$
\|Q_0(z) - Q_0(u_n^1) \| < \gamma_1,
$$
for any $z \in B(u_n^{1}, 2r_n^1)$ and for any $n$. Here $\gamma_1 = \gamma (A, d_0,  \eta_1,  d)$ where $\gamma $ is given in Theorem \ref{cons}. Note that, since $B(u_n^{1}, 2r_n^1) \subset D_1 \cup D_2$, we have, for any $n$,
$$
\eta_{1} |h|^2 \le  \langle Q_0 (z) h,h \rangle \le \frac{1}{\eta_{1}} |h|^2,\;\;\; h \in
\R^{d_0},\;\; z \in B(u_n^{1}, 2r_n^1).
$$
Similarly, we can find a finite sequence $(u_n^{k}) \subset D_k $ such that  $D_k\subset
\cup_{n} B(u_n^{k}, r_n^k) $ with $0<r_n^k <1/2$ and
$$
\|Q_0(z) - Q_0(u_n^k) \| < \gamma_k,
$$
for any $z \in B(u_n^{k}, 2r_n^k)$, where $\gamma_k = \gamma (A, d_0,  \eta_k,  d)$   from Theorem \ref{cons}. Moreover, for any $n$,
\begin{equation*}
  \eta_{k} |h|^2 \le  \langle Q_0 (z) h,h \rangle \le \frac{1}{\eta_{k}} |h|^2,\;\;\; h \in
\R^{d_0},\;\; z \in B(u_n^{k}, 2r_n^k),\;\; k \ge 2.
\end{equation*}
 By considering all  the previous balls $B(u_n^{k}, r_n^k)$ and $(\eta_k)$
 we get the desired covering $\{ B(z_j, \delta_j) \}_{j \ge 1}$ which verifies \eqref{ett}.

  The balls
 $\{ B(z_j,  \delta_j) \}_{j \ge 1}$ give the covering $\{ U_j\}_{j \ge 1}$ used in  Theorem
 \ref{uni1}.

\hh \textit{II Step.}
   We define suitable operators ${{ {\mathcal A}}}_j$ such that
 \begin{align} \label{bo1}
 {{\mathcal A}}_j f(z) = { {\mathcal L}} f(z),\;\; z \in U_j=B(z_j,  \delta_j),\;\; f \in C^2_K(\R^d).
 \end{align}
 We fix $j \ge 1$ and consider
  $\rho_j \in C^{\infty}_K(\R^d)$ with $0 \le \rho_{j} \le 1$,
 $\rho_{j} =1$ in $B(z_j, \delta_j)$ and  $\rho_j =0$ outside $B(z_j, 2
 \delta_j )$. Now set
 $$
 Q^j_0(z) :=  \rho_{j}(z) Q_0(z) +  (1 - \rho_{j}(z))Q_0(z_j),\;\; z\in \R^d.
 $$
 We see that by \eqref{ett}
 \begin{gather} \label{ett3}
  \lan Q^j_0(z) h, h \ran =  \rho_{j}(z) \lan Q_0(z) h, h \ran
   +  (1 - \rho_{j}(z)) \lan Q_0(z_j) h , h \ran \ge \eta_j |h|^2,
 \\
\nonumber \lan Q^j_0(z) h, h \ran  \le \frac{1}{\eta_j} |h|^2, \;\; z \in \R^d,\;\; h \in \R^{d_0}.
 \end{gather}
  Moreover $Q^j_0(z) = Q_0(z)$, $z \in U_j$, and
 \begin{equation}\label{ett4}
 \| Q^j_0(z) - Q_0(z_j) \| < \gamma_j,
 \end{equation}
 for any $z \in \R^d$; recall that $\gamma_j = \gamma (A, d_0,  \eta_j,  d)$ and $\gamma $ is given in Theorem \ref{cons}.
 We finish the step by  defining
 $$
 {{\mathcal A}}_j f(z) = \frac{1}{2} \text{Tr}(Q^j_0(z)D^2_x f(z)) + \langle Az , Df(z) \rangle + \langle b_j(z), D_x f(z) \rangle,\;\; z \in \R^d,
 $$
  $f \in C^2_K (\R^d)$. Here $b_j = b_0 \cdot 1_{B(z_j,\delta_j)}$ ($1_{B(z_j,\delta_j)}$
  is the indicator function of  ${B(z_j,\delta_j)}$).

 \medskip \noindent \textit{III Step.} We show that the martingale problem for each ${\mathcal A}_j$ is well-posed. This will allow to apply Theorem \ref{uni1} and will finish the proof.

\smallskip First  note that the martingale problem for each $\L^{(j)}$,
$$
   \L^{(j)} f(z) = \frac{1}{2} \text{Tr}(Q^j_0(z)D^2_x f(z)) + \langle Az , Df(z) \rangle,
 $$
  $f \in C^2_K (\R^d)$, $z \in \R^d$, is well-posed by Theorem \ref{cons}. Indeed, using \eqref{ett3} and
  \eqref{ett4}, we see that the assumptions of Theorem \ref{cons} are satisfied.

Let us fix $j \ge 1$. By Theorem \ref{one} to prove the  well-posedness of the martingale problem associated to ${\mathcal A}_j$, it is enough to show that, for any $z \in \R^d$, the martingale problem for $({{\mathcal A}}_j, \delta_z)$ is well-posed.
Let   $z_0 \in \R^d$ and
 consider the SDE
\begin{equation}\label{by1}
    dZ_t
 = A Z_t dt
 \, + \,
\begin{pmatrix}
 b_j(Z_t) \\
0
\end{pmatrix}
 dt \, + \,
 \begin{pmatrix}
 \sqrt{Q_0^j(Z_t)} &
 0\\
 0 & 0
\end{pmatrix}   dW_t,\, \;\; Z_0 =z_0,
\end{equation}
where $\sqrt{Q_0^j(z)}$   denotes   the unique symmetric $d_0 \times d_0$ square root of $Q_0^j(z)$; note that $\sqrt{Q_0^j(z)}$
  is  a continuous  function of $z$. Moreover $W = (W_t)$ is a standard Wiener process with values in $\R^d$.
By Theorems \ref{iki} and \ref{eki} it is enough to prove the well-posedness of
the SDE \eqref{by1}.

Since the martingale problem for $\L^{(j)}$ is well-posed,   we know the well-posedness of the SDE
 \begin{equation}\label{by11}
    dZ_t
 = A Z_t dt
 \, + \,
 \begin{pmatrix}
 \sqrt{Q_0^j(Z_t)} & 0\\
 0 &
 0
\end{pmatrix}   dW_t,\, \;\; Z_0 =z_0.
\end{equation}
 An application of the Girsanov theorem
 (see Theorem IV.4.2 in \cite{IW})  allows to deduce that there exists a unique weak solution to
\begin{align} \label{by2}
     dZ_t
 = \Big ( A Z_t
 \, + \,
\begin{pmatrix}
 \sqrt{Q_0^j(Z_t)} &
  0\\
 0 &
 0
\end{pmatrix} a(Z_t) \Big)
 dt\end{align} $$ \, + \,
 \begin{pmatrix}
 \sqrt{Q_0^j(Z_t)} & 0\\
 0 &
 0
\end{pmatrix}   dW_t,\, \;\; Z_0 =z_0,
$$
 if $a: \R^d \to \R^d$ is any Borel and bounded function. By defining
 $$a(z) = \begin{pmatrix}
 ({Q_0^j(z)})^{-1/2} \, b_j(z) \\
0
\end{pmatrix}, \;\;\; z \in \R^d,$$ we obtain  that $a$ is bounded by \eqref{ett3} and moreover equation
  \eqref{by2} becomes equation \eqref{by1}. This proves the assertion and completes the proof.
 \qed

\begin{remark} \label{stro}
{\em   In the proof of the previous result we can not apply directly the results in Section 6.6 of  \cite{SV79} instead of Theorem \ref{uni1}.
Indeed the mentioned results in \cite{SV79} would require  to truncate both coefficients $ A z $ and $Q_0(z)$ on balls   in order to deal with  diffusions  with bounded coefficients.
The problem is that  if we truncate in the previous way and then   consider  the
 truncated mapping of $z \mapsto Az$ it becomes difficult   to prove the analytic  regularity results of  Sections 3.1 which are needed to prove well-posedness.}
\end{remark}

\noindent \textbf{ Proof of Theorem \ref{main2}.}  If we assume in addition that $b_0 :  \R^{d}  \to \R^{d_0}$ is continuous then the assertion follows immediately from Theorems \ref{sko} and \ref{main1}.

In the general case when $b_0$  is only locally bounded we have to use also Corollary \ref{exi2}. To this purpose we consider  the sequence $\{ U_{k}\}_{k \ge 1}$, where   
 $U_k = B(0, k)$ (the open ball of center $0$  and radius $k$). We show that there exist  linear operators ${\cal M}_{k}$ with common  domain $C^{2}_K$ such that 

\smallskip
\noindent (i)  for any $k \ge 1$, $f \in C^2_K $, we have
$
 {\cal M}_{k}  f(z) $ $= {\cal L}^{} f(z),$ $z \in U_k;
$

\noindent (ii)  the martingale problem for each ${\cal M}_{k} $ is well-posed;

\noindent  (iii) if $Z^k= Z^{k,z}$ is a (unique in law) martingale solution
for  $( {\cal M}_{k} , \delta_z)$, $z \in \R^d$,
 defined on a probability space $(\Omega_k,{\cal F}_k, P_k )$ and  
$
\tau_k =\tau_k^z  = \inf \{ t \ge 0 \; :\; Z_t^k \not \in U_k \}
$    we have, for any $t>0,$ $z \in \R^d$,
\begin{equation}\label{stop214}
    \lim_{k \to \infty} P_k (\tau_k \le t) =0.
\end{equation}
Once (i), (ii) and (iii) are proved, the assertion follows. 

\smallskip 
 For any $k \ge 1$ define  $\psi_k \in C^{\infty}_K(\R^d)$ such that $0 \le \psi_k \le 1$, $\psi_k(z) =1$ for $|z| \le k$ and $\psi_k(z) =0$ for $|z| \ge 2k$.  Since $z \mapsto Q_0(z)$ is continuous,  the minimum and the maximum eigenvalue of $Q_0(z)$ depend  continuously on $z$. Therefore,  
 there exists $\eta_k \in (0, 1)$ such that
\begin{equation*}
  \eta_{k} |h|^2 \le  \langle Q_0 (z) h,h \rangle \le \frac{1}{\eta_{k}} 
|h|^2,\;\;\; h \in
\R^{d_0},\;\; z \in \bar{U_{k}},\;\; k \ge 1.
\end{equation*}
Define the $d_0 \times d_0$-matrix 
$$
Q_0^k(z) = \psi_k (z) Q_0(z) + (1-\psi_k (z)) I_0,\;\;\; z \in \R^d,
$$
where $I_0$ is the $d_0 \times d_0$-identity matrix. 
It is clear  $Q_0^k(z) = Q_0(z),$ $ |z| \le k.$ Moreover,
\begin{align*}
 \eta_k |h|^2 \le \lan Q_0^k(z) h, h \ran \le \frac{1}{\eta_k} |h|^2, 
\;\; h \in \R^{d_0}, \;\; z \in \R^d,\;\; k \ge 1.
\end{align*}
Then, we set 
$$
{\cal M}_{k} f(z) = \frac{1}{2} \Tr(Q_0^k(z) D^2_x f(z)) + 
\lan b_k(z), D_x f(z) \ran
+ \lan A z , D f(z)\ran, \;\; \; f \in C^2_K,
$$
 $ z \in \R^d,$ where  $b_k = b_0 \cdot \psi_k$, $k \ge 1$.
By construction each ${\cal M}_{k} $ verifies condition (i). Moreover,
by using the Girsanov theorem as  in III Step of the proof of Theorem \ref{main1}, we know that 
 the martingale problem for ${\cal M}_{k} $ is well-posed if the martingale 
problem for ${\cal B}_{k}$ is well-posed,
$$
{\cal B}_{k} f(z) = \frac{1}{2} \Tr(Q_0^k(z) D^2_x f(z)) + 
 \lan A z , D f(z)\ran, \;\; \; f \in C^2_K.
$$
Let us fix $k \ge 1$. Using that the coefficients of ${\cal B}_{k} $ are continuous and that they growth at most linearly, we know by Theorem \ref{sko} that, for any $z \in \R^d$,
there exists a martingale solution for $({\cal B}_{k}, \delta_z)$. Since  ${\cal B}_{k}$ verifies Hypothesis \ref{hy} with $b_0=0$, we can apply  Theorem \ref{main1}
and obtain that the  martingale 
problem for ${\cal B}_{k}$ is well-posed. Thus  condition (ii) is verified for each ${\cal M}_{k} $.

 It remains to check (iii). Let us fix $z_0 \in \R^d$ and denote by $Z^k = (Z^k_t)$ a solution to the martingale problem for $({\cal M}_{k} , \delta_{z_0})$  defined on  $(\Omega_{k}, {\cal F}_{k}, P_{k})$.  
Let   $k $ large enough such that $z_0 \in U_k$ and consider the Lyapunov function $\phi$  (see \eqref{lia}). It is easy to see that there exists $\phi_k \in C^2_K(\R^d)$ such that $\phi(z) = \phi_k(z)$, $z \in U_k$.
By the optional stopping theorem  we know that
$$
 \phi_k (Z^k_{t \wedge \tau_k}) - \int_0^{t \wedge \tau_k} {\cal M}_k \phi_k(Z_s^k)ds
$$
is a martingale. Denoting by $E_{k} $ the expectation with respect to $P_{k}$, we find,  for $t \ge 0,$
$$
E_{k} [\phi (Z^k_{t \wedge \tau_k})] = \phi(z_0) +  E_{k} 
\Big[ \int_0^{t \wedge \tau_k} \L \phi (Z_s^k)ds \Big] \le \phi(z_0) +  C \int_0^{t } E_{k} [\phi (Z_{s \wedge \tau_k}^k)]ds.
$$
By the Gronwall lemma we get
$
E_{k} [\phi (Z^k_{t \wedge \tau_k})\,  1_{\{ \tau_k \le t\} }] \le \phi(z_0) \, e^{Ct},
$
so that $$\min_{|y|=k} \{ \phi(y)\} \cdot  P_{k} ( \tau_k \le t) \le \phi(z_0) e^{Ct},$$ $t \ge 0$. Since $\phi \to \infty$ as $|z| \to \infty$ we obtain \eqref{stop214}. The proof is complete.
\qed

\vskip 8mm

\noindent \textbf{Acknowledgement.}
The author would like to thank the  referees
 for their
 useful remarks and suggestions
 which helped to  improve the paper.

\vskip 8mm


\appendix

\section{
Appendix: the  localization principle
for martingale problems
}

The   localization principle  introduced by  Stroock and Varadhan (see \cite{SV69} and
 \cite{SV79})
  says, roughly speaking, that to prove uniqueness in law
  it suffices to show that each starting point has a
 neighbourhood on which the diffusion coefficients equal other coefficients for which uniqueness
holds
  (see also \cite{letta, Ko}).
 Martingale problems and localization principle  have been extensively investigated in Chapter  4 of \cite{EK}
 in the setting of
 a complete and separable  metric space $E$. This generality allows
  applications of the martingale problem
   to branching processes
  (see Chapter 9 in \cite{EK}) and to SPDEs
 (see, for instance, \cite{DD, DZ} and the references therein).

\smallskip
{\sl In this appendix  we present some extensions and
  modifications of  theorems  given in Sections 4.5 and 4.6 of \cite{EK}.
  Our main results are in Section A.3 (see in particular Theorem \ref{key} and  Lemma \ref{stop}).
 As a consequence we get
 the localization principle (see Theorem \ref{uni1}) which is an
 extension of Theorem 4.6.2   in \cite{EK} and of Theorem 6.6.1 in \cite{SV79}. }
\smallskip

     Unlike Sections 4.5 and 4.6 of \cite{EK} which  mainly deal with  c\`adl\`ag martingale solutions here we always work with martingale solutions with continuous paths. It is not straightforward to extend results in \cite{EK} about the localization principle from c\`adl\`ag to  continuous martingale solutions; see in particular Lemma 4.5.16 in \cite{EK}.
On the other hand,
proving well-posedness   in the class of  c\`adl\`ag solutions
can be more difficult than in the class of continuous solutions. This is particularly important for martingale problems related to SPDEs with Wiener noise  (see, for instance, the recent paper \cite{bass} where  infinite dimensional OU type processes are considered). Our localization principle can also be applied to  such equations.


    Another difference  with respect to \cite{EK},
  is that we always
assume  that the linear operator $A$ appearing in the martingale problem
 is countably pointwise determined
  (see Hypothesis \ref{bpt}).   This assumption  is usually satisfied in applications and  allows to improve some results from \cite{EK} (see, in particular, Section A.2).

\subsection{Basic definitions} In this appendix $E$ will denote a \textit{complete and separable metric space} endowed with its $\sigma$-algebra of Borel sets ${\cal B}(E) $. The  space of all real bounded and Borel functions on $E$
is indicated with $B_b(E)$. It is a Banach space with the supremum norm $\| \cdot \|_{\infty}$.
Its closed subspace $C_b(E)$ is the space of all real bounded and continuous functions on $E$.  We will also consider  the space  $C_E[0, \infty)$ of all  continuous functions from $[0, \infty)$ into $E$.  This is a
 complete and separable  metric space endowed
with the metric of uniform convergence on compact sets of $[0,
\infty)$. In addition ${\cal P}(E)$
denotes the metric space   of
all Borel probability measures  on $E$
endowed with the   Prokhorov
metric which induces the weak convergence of measures. It  is a
complete and separable metric space (see Chapter 3 in \cite{EK}).
Its Borel $\sigma$-algebra is denoted by ${\cal B}({\cal P}(E))$.

 \hh Let us fix a linear operator $A$ with domain $D(A) \subset C_b(E)$
 taking values in $B_b(E)$, i.e.,
\begin{align} \label{lin}
  A: D(A) \subset C_b(E) \to B_b(E) \;\; \text {is linear.}
\end{align}
 Let $\mu \in {\cal P}(E)$.
An $E$-valued stochastic process $X = (X_t)= (X_t)_{  t \ge 0 }$
defined on some probability space $(\Omega, {\cal F}, P)$
with continuous trajectories is a \textit { solution of the martingale
problem for $(A, \mu)$} if,
 for any $f \in D(A)$,
\begin{equation}\label{mart}
    M_t(f) = f(X_t) - \int_0^t A f(X_s) ds, \;\; t \ge 0, \;\; \text{is a
    martingale}
\end{equation}
(with respect to the natural filtration $({\cal F}_t^X)$,
 where ${\cal F}_t^X = \sigma(X_s \, : \, 0 \le s \le t)$ is the  $\sigma$-algebra generated by the random variables
$X_s$, $0 \le s \le t$),
  and moreover, the law of $X_0$ is $\mu.$

\hh Comparing with   \cite{EK} we only consider solutions $X$ to the
 $C_E[0, \infty)$-martingale  problem for $(A,\mu)$ (see also Remark \ref{serve}).

\hh It is also convenient to call a Borel probability $P$ on $C_E[0, \infty)$
(i.e., $P \in {\cal P}(C_E[0, \infty))$)
a {\it (probability) solution of the martingale problem} for $(A, \mu)$ if the {\it canonical process}
 $X = (X_t)$ defined on $(C_E[0, \infty),$ $ {\cal B}(C_E[0, \infty)), P )$ by
  \begin{equation}\label{can}
    X_t(\omega) = \omega(t),\;\;\; \omega \in C_E[0, \infty), \;\; t \ge 0,
 \end{equation}
 is a solution of  the martingale problem for $(A, \mu)$.

  The martingale  property \eqref{mart} only concerns the finite dimensional distribution of $X$.
In fact it is equivalent to the following property: {\it for arbitrary $0 \le t_1 < \ldots < t_n < t_{n+1}$, $f \in D(A)$ and  arbitrary $h_1, \ldots, h_n
\in C_b (E)$, we have}
\begin{equation}\label{equi}
E \big[ \big( M_{t_{n+1}}(f) - M_{t_{n}}(f) \big) \cdot \prod_{k=1}^n h_k(X_{t_k})
\big] =0.
\end{equation}
Hence $X$ is a martingale solution for $(A, \mu )$ if and only if its law
 on $(C_E[0, \infty),$ $ {\cal B}(C_E[0, \infty)))$ is a martingale
 solution for $(A, \mu)$.

\begin{remark} \label{serve} {\em  We give additional comments motivated by
  \cite{EK}.

i) We have required that a solution has sample paths in $C_E[0, \infty)$.
 On the other hand as in \cite{EK}
one can also consider   martingale solutions $X$  which have c\`adl\`ag trajectories, that is,
 they have sample paths in $D_E[0, \infty)$ ($D_E[0, \infty)$ denotes  the complete and separable  metric space   of all  c\`adl\`ag functions from $[0, \infty)$ into $E$  endowed with the Skorokhod metric).

 The book \cite{EK} treats even more general
   martingale solutions $X$  without c\`adl\`ag trajectories. Moreover in \cite{EK}
   the reference filtration $({\cal G}_t)$ can be larger than $({\cal F}_t^X)$; this allows to obtain the Markov property with respect to $({\cal G}_t)$
    when the martingale problem is well-posed.

 ii) Recall that, for any $x \in E$, $\delta_x \in {\mathcal P}(E)$ is defined
by
 \begin{align} \label{delta}
 \delta_x(A)= 1_A(x),
 \;\;\; x \in E, \;\; A \in {\cal B}(E).
 \end{align}
(where $1_A(x) =1$ if $x \in A$ and $1_A(x) =0$ if $x \not \in A$). According to  Theorem 4.3.5 in \cite{EK}  if
there exists a solution $X_x$
 of the martingale problem for $(A, \delta_x)$ for any $x \in E$ then $A$ is dissipative, i.e., $
\lambda \| f\|_{\infty}$ $ \le \| \lambda f - Af\|_{\infty},$ $ \lambda>0,$ $ f \in D(A).
$
Further  relations between
   the martingale problem and  semigroup theory of linear operators are investigated in \cite{EK}.
   }
\end{remark}
\begin{definition} \label{baa}
Let $\mu \in {\cal P}(E)$.  We say that  uniqueness holds for the martingale  problem for $(A, \mu)$
if all the  solutions $X$  have the same finite dimensional distributions
(i.e., all the solutions  $X$   have the same law on
 $C_E[0, \infty)$, i.e., all (probability) martingale solutions $P$ coincide on
  ${\cal B}(C_E[0, \infty)))$.

The martingale problem for $(A, \mu)$ is well-posed if there exists a martingale solution
 for $(A, \mu)$ and, moreover,
 uniqueness holds  for  the martingale problem for $(A, \mu)$.

 Finally, the martingale problem for $A$ is well-posed if the martingale problem for $(A, \mu)$ is well-posed for any $\mu \in {\cal P}(E)$.
\end{definition}
Next  we  consider boundedly and pointwise convergence for
multisequences of functions similarly to \cite{EK}, page 111, and
\cite{DT}.
 \begin{hypothesis} \label{bpt}
{\em  A linear operator $A: D(A) \subset C_b(E) \to B_b(E)$ is countably pointwise determined (c.p.d.)
  if there exists a countable subset
 $H_0 \subset D(A)$ such that for any $f \in D(A)$ there exists an $m$-sequence of functions $(f_{n_1, \ldots, n_m}) \subset H_0$, $(n_1, \ldots, n_m) \in \mathbb{N}^m$, $m \ge 1$,
 such that   $(f_{n_1, \ldots, n_m})$
  and $(Af_{n_1, \ldots, n_m})$ converge boundedly and pointwise respectively to $f $
 and $Af$. This means that  there exists $C>0$
 such that
$ \| f_{n_1, \ldots, n_m}\|_{\infty} $ $+ \| Af_{n_1, \ldots, n_m}\|_{\infty} \le C $,  for any $(n_1, \ldots, n_m) \in \N^m$,
 and moreover }
$$
\begin{array}{c}
\lim_{n_1 \to \infty} \ldots (\lim_{n_{m-1} \to \infty}(\lim_{n_m \to \infty} f_{n_1, \ldots, n_m}(x))) = f(x), \;\; x \in E.
\\
 \lim_{n_1 \to \infty} \ldots (\lim_{n_{m-1} \to \infty}(\lim_{n_m \to \infty} Af_{n_1, \ldots, n_m}(x))) = Af(x), \;\; x \in E. \qed
\end{array}
$$
\end{hypothesis}
In particular  $A$ is c.p.d.    if there exists a separable subspace $M$
 of $C_b(E)$  such that $\{(f, Af)\}_{f \in D(A)} $ $\subset M \times M$.

{\sl It is easy to verify that if Hypothesis \ref{bpt} holds for $A$ then it is enough to check the martingale property \eqref{mart} only for  $f \in H_0$ in order to have a martingale solution.
}

\subsection{Preliminary results}

 Results and arguments of this section
 are quite similar to those
given in Chapter 6 of \cite{SV79}
   (see also \cite{Ka, KS}) even if here we are in the general setting of
   martingale solutions with values in a Polish space.
 We include self-contained  proofs for the sake of completeness.

Assuming Hypothesis \ref{bpt} to prove well-posedness
 we only have to check that the martingale problem  is
 well-posed for any initial distribution $\delta_x$, $x \in E$
 (see \eqref{delta}).

The first  result deals with uniqueness of the martingale problem
for $(A, \delta_x)$   for any $x \in E$ (cf. Theorem 6.2.3 in
\cite{SV79} and Theorem 4.27 in \cite{KS})). It is  a variant    of
Theorem 4.4.6 in \cite{EK} which  considers the case when, starting
from {\it any initial distribution} $\mu \in {\cal P}(E)$, any two
martingale solutions have the same marginals.
\begin{theorem}
\label{counta} Suppose that the operator $A$ satisfies Hypothesis \ref{bpt}.
 Suppose that, for any $x \in E$, any two (probability) martingale
  solutions $P^x_1$ and $P^x_2$ for
$(A, \delta_x)$ have the same one dimensional
 marginal distributions, i.e.,
\begin{equation}\label{marg}
    P_1^x (X_t \in B) =  P_2^x (X_t \in B),\;\; t \ge 0,\;
    \; B \in {\cal B}(E),
\end{equation}
where $(X_t)$ denotes the canonical process in \eqref{can}.
 Then, for any $x \in E$, there exists at most one
   martingale solution for  $(A, \delta_x)$.
\end{theorem}
\begin{proof} Let $P^x_1 = P_1$ and $P_2^x = P_2$ and set
 $\Omega = C_E[0, \infty)$ endowed with the Borel
$\sigma$-algebra ${\cal F} =  {\cal B}(C_E[0, \infty))$.  Take any
 sequence $(t_k) \subset [0, \infty)$, $0 \le t_1 < \ldots < t_n < \ldots$.
 It is enough to show that, for any $n \ge 1$,
$P_1$ and $P_2$ coincide on the $\sigma$-algebra  $\sigma(X_{t_1}, \ldots, X_{t_n}) $
  generated by $X_{t_1}$,  $\ldots, X_{t_n}$.
  To show this we use induction on $n$.
 For $n=1$ the assertion follows from  \eqref{marg}.
 We assume that the assertion holds
for $n-1$ with $n \ge 2$ and  prove it for $n$. Set
$$
{\cal G} = \sigma(X_{t_1}, \ldots, X_{t_{n-1}}).
$$
We know that $P_1$ and $P_2$ coincide on $\cal G$. Since $\Omega = C_E[0, \infty)$
 is a complete and separable metric space, by applying Theorem 3.18, page 307
 in \cite{KS} there exists
 a regular conditional probability $Q_1^{\omega}$ for $P_1$ given $\cal G$; this satisfies:

 a) for any $\omega \in \Omega$, $Q_1^{\omega}$ is a probability on $(\Omega, {\cal F})$;

b) for any $A \in {\cal F}$, the map: $\omega \mapsto Q_1^{\omega}(A)$ is ${\cal G}$-measurable;

c) for any $A \in {\cal F}$, $Q_1^{\omega}(A)= P_1(A / {\cal G})(\omega)
 := E^{P_1}[1_A / {\cal G}](\omega)$,
  $P_1$-a.s..

\hh
By $E^{P_1}[1_A / {\cal G}]$ we have indicated the conditional expectation of $1_A$
 with respect to ${\cal G}$ in $(\Omega, {\cal F}, P_1)$.
Moreover, since ${\cal G}$ is countable determined (i.e., there exists a countable set ${\cal M} \subset {\cal G}$ such that whenever two probabilities agree on $\cal M$ they also agree on $\cal G$) we also have that there exists $N' \in {\cal G} $ with $P_1(N')=0$ and
\begin{align}\label{con}
Q_1^{\omega} (A) = 1_A(\omega), \;\;\; A \in {\cal G},  \; \omega \not \in N'.
\end{align}
Now the proof continues  in two  steps.

\hh {\it I Step.} \textit{We  show that there exists a $P_1$-null set $N_1 \in {\cal G}$ such that,
 for any $\omega \not \in N_1$, the probability measure
  $R_1^{\omega} = Q_1^{\omega} \circ \theta_{t_{n-1}}^{-1}$, i.e.,
$$
R_1^{\omega}(B) = Q_1^{\omega} \big ( (\theta_{t_{n-1}})^{-1} (B ) \big),\;\;\; B \in {\cal F},
$$
solves the martingale problem for $(A, \delta_{\omega(t_{n-1})})$.}

Here $
\theta_{t_{n-1}} : \Omega $ $ \to \Omega$ is a shift operator, i.e., $\theta_{t_{n-1}}(\omega)(s)
 = \omega(s+ {t_{n-1}})$, $s \ge 0$.
  It is clear by \eqref{con} that there exists a  $P_1$-null set $N'
 \in {\cal G}$ such that for any $\omega \not \in N'$,
$$
R_1^{\omega} (\omega' \in \Omega\, : \, \omega'(0)= \omega(t_{n-1}))
 = Q_1^{\omega} (\omega' \in \Omega \, : \, \omega'(t_{n-1}+0)=
 \omega(t_{n-1})) =1.
$$
To prove the martingale property \eqref{equi}
we first  introduce the
family $\cal S$
   of all finite intersections of open balls $B(x_i, 1/k)\subset E$,
   where $k \ge 1$ and  $x_i \in E_0$ with $E_0$ a fixed
    countable and dense subset of
   $E$, and then
  consider the countable set $\Gamma$ of bounded
  random variables $\eta: \Omega \to \R$ of the form
 \begin{equation} \label{etaa}
\begin{array}{c}
\eta = \big( M_{s_{m+1}}(f) - M_{s_{m}}(f) \big) \cdot \prod_{k=1}^m
h_k(X_{s_k})
\\
\nonumber = \Big( f(X_{s_{m+1}}) - f(X_{s_{m}}) -
\int_{s_m}^{s_{m+1}} Af(X_r) dr  \Big) \cdot \prod_{k=1}^m
h_k(X_{s_k}),
\end{array}
\end{equation}
where  $f \in H_0$ (see Hypothesis \ref{bpt}),
 $0 \le s_1 < \ldots < s_m < s_{m+1}$, $m \ge 1$,
 are arbitrary rational numbers,
 $h_k$ are indicator functions of sets
 in ${\cal S}$ and $(X_t)$ is the canonical process.
  By using a monotone class argument it is not difficult to see that
$R_1^{\omega}$ solves the martingale problem for $(A,
\delta_{\omega(t_{n-1})})$ if and only if
 $\int_{\Omega} \eta (\omega')R_1^{\omega} (d \omega') =0$
 for any $\eta \in
 \Gamma.$

 Therefore the claim follows if we  prove that for a fixed $\eta \in \Gamma$
there exists a $P_1$-null set $N \in {\cal G}$ (possibly depending on $\eta$)
 such that  for any $\omega \not \in N$,
\begin{equation*}
    \int_{\Omega} \eta (\omega')R_1^{\omega} (d \omega') =0.
\end{equation*}
To show that the ${\cal G}$-measurable random variable $\omega \mapsto
 \int_{\Omega} \eta (\omega')R_1^{\omega} (d \omega')$ is 0, $P_1$-a.s., it is enough to prove that, for any $G \in {\cal G} = \sigma(X_{t_1}, \ldots, X_{t_{n-1}})$,
\begin{equation*}
 \int_{\Omega} \Big[
 1_G(\omega)
   \int_{\Omega} \eta (\omega')R_1^{\omega} (d \omega')  \Big] P_1 (d \omega) =0.
\end{equation*}
  We have
$$\begin{array}{c}
\int_{\Omega} \Big[
 1_G (\omega)
  \int_{\Omega} \eta (\omega')R_1^{\omega} (d \omega')  \Big] P_1 (d \omega)
\\
=\int_{\Omega} \Big[
 1_G (\omega)
  \int_{\Omega} \Big(  \big( M_{s_{m+1}+ t_{n-1}}(f) - M_{s_{m}+ t_{n-1}}(f)
   \big) \cdot
 \\
\cdot \prod_{k=1}^m h_k(X_{s_k + t_{n-1}})  \Big)(\omega')
 Q_1^{\omega} (d \omega')  \Big] P_1 (d \omega)
  \\ \\
 =  E^{P_1} \big[
 1_G  \, E^{P_1} [  \eta  \circ \theta_{t_{n-1}} /  {\cal G}] \big]
 =  E^{P_1} \big[
  \, E^{P_1} [  (\eta  \circ \theta_{t_{n-1}})
  1_G /  {\cal G}] \big] \\ =   E^{P_1} [  \big( M_{s_{m+1}+ t_{n-1}}(f) - M_{s_{m}+ t_{n-1}}(f)
   \big)
\cdot \prod_{k=1}^m h_k(X_{s_k + t_{n-1}}) \cdot  1_G ] =0
 \end{array}$$
(in the last passage we have used that $P_1$ is a martingale solution).

\hh {\it II Step. } {\it We show that $P_1$ and $P_2$ coincide on
 $\sigma(X_{t_1}, \ldots, X_{t_{n}})$}.

Repeating the previous step for the measure $P_2$ we define
  $Q_2^{\omega}$ (the regular conditional probability for $P_2$ given $\cal G$)
  and  $R_2^{\omega} = Q_2^{\omega} \circ \theta^{-1}_{t_{n-1}}$. We find that there exists  a $P_2$-null set $N_2 \in {\cal G}$ such that
 for any $\omega \not \in N_2$, the probability measure
$
R_2^{\omega} $
solves the martingale problem for $(A, \delta_{\omega(t_{n-1})})$.

Since $P_1$ and $P_2$ coincide on ${\cal G}$, the set $N'= N_1 \cup N_2$
verifies $P_k(N')=0$, $k=1,2$.  By hypothesis, for any $\omega \not \in N'$
we know that $R_1^{\omega}$ and $R_2^{\omega}$ have the same one-dimensional
 marginals. Therefore, for any $A \in {\cal B}(E^{n-1})$, $B \in {\cal B}(E)$, we find
$$
\begin{array}{c}
P_1 \big( \omega \in \Omega \, : \, (\omega(t_1), \ldots, \omega(t_{n-1})) \in A, \omega(t_n) \in B \big)\\
= E^{P_1} [1_{ \{ \omega : (\omega(t_1), \ldots, \omega(t_{n-1})) \in A \}} \, R_1^{\omega}
(\omega \in \Omega \, : \, \omega(t_n - t_{n-1}) \in B)  ]
\\
= E^{P_1} [1_{ \{ \omega : (\omega(t_1), \ldots, \omega(t_{n-1})) \in A \}} \, R_2^{\omega}
(\omega \in \Omega \, : \, \omega(t_n - t_{n-1}) \in B)  ].
\end{array}
$$
Since $\omega \mapsto R_2^{\omega}
(\omega \in \Omega \, : \, \omega(t_n - t_{n-1}) \in B) $ is $\cal G$-measurable
 and $P_1= P_2$ on ${\cal G}$ we
get
$$
\begin{array}{c}
P_1 \big( \omega \in \Omega \, : \, (\omega(t_1), \ldots,
\omega(t_{n-1})) \in A, \omega(t_n) \in B \big)
\\
= E^{P_2} [1_{ \{ \omega : (\omega(t_1), \ldots, \omega(t_{n-1}))
\in A \}} \, R_2^{\omega} (\omega \in \Omega \, : \, \omega(t_n -
t_{n-1}) \in B)  ]\\
 = P_2 \big( \omega \in \Omega \, : \,
(\omega(t_1), \ldots, \omega(t_{n-1})) \in A, \omega(t_n) \in B
\big).
\end{array}
$$
This finishes the proof.
\end{proof}

\hh Recall that a family of measures    $(P^x)=(P^x)_{x \in E}
\subset {\cal P}(C_E[0, \infty))$ \textit{ depends measurably on $x$} (cf. Lemma 1.40 in \cite{Ka}) if for
any $B \in {\cal B}(C_E[0, \infty))$, the mapping:
\begin{align} \label{mes1}
 x \mapsto P^x(B) \;\;
 \text{is measurable from $E$ into $[0,1]$.}
\end{align}
 Suppose that, for any $x \in E$,
 there exists a martingale solution $P^x$ on
  ${\cal B}(C_E[0, \infty))$
  for $(A, \delta_x)$.
If  $(P^x)$ depends measurably on $x$ then it is easy to check that, for any initial
distribution $\mu \in {\cal P}(E)$, there exists a  martingale solution $P^{\mu}$ for $(A, \mu)$ which is given by
\begin{equation}\label{es}
P^{\mu}(B) = \int_E P^x(B) \mu (dx), \;\;\; B \in {\cal B}(C_E[0,
\infty)).
\end{equation}
Usually, $(P^x)$ depends measurably on $x$ if one provides a
constructive proof for existence of   martingale solutions. On the
other hand, the
next theorem shows that uniqueness implies this measurability property.
  This  result
    is a kind of extension  of  Theorem 4.4.6  in \cite{EK}
 (in fact in \cite{EK} it is required
  that the martingale problem is well-posed for any initial $\mu \in {\cal P}(E)$).
\begin{theorem} \label{one}
 Suppose that   $A$ satisfies Hypothesis \ref{bpt}.
Suppose that, for any $x \in E$, there exists a unique (probability) martingale
solution
 $P^x$ for $(A, \delta_x)$.

 Then $(P^x) $ depends measurably on $x$ and for
any initial distribution $\mu \in {\cal P}(E)$ there exists a unique (probability)
martingale solution $P^{\mu}$ given by \eqref{es}. In particular
   the martingale problem for $A$ is well-posed.
 \end{theorem}
\begin{proof} We combine ideas from the proofs of  Theorem 21.10
 in \cite{Ka} and that  of  Theorem 4.4.6 in \cite{EK}.
 In the sequel   $\Omega =
C_E[0, \infty)$ and we denote with $\cal F$ its Borel $\sigma$-algebra. Recall that   ${\cal P} (E)$ and ${\cal P}(\Omega)$ are  complete and separable metric spaces
  with the  Prokhorov
 metric.

\hh \textit{I Step.}  We consider  the countable  family $\Gamma$ of
random variables $\eta$  defined in \eqref{etaa} by means of the canonical process $(X_t)$.  Recall that by  a
monotone class argument,  $P \in {\cal
P}(\Omega)$ is a martingale solution for $(A, \delta_x)$ if and
only if $P(X_0 \in A)= $ $P(X_0^{-1} (A))= \delta_x(A)$, $A \in
{\cal B}(E)$, and
\begin{equation}\label{chec} \int_{\Omega} \eta(\omega) P(d \omega)=0,
\;\;\; \eta \in \Gamma.
\end{equation}
\textit{II Step.} We prove that the set $(P^x)_{x \in E}$ of all
martingale solutions (each $P^x$ is the unique martingale solution for $(A,
\delta_x)$) belongs to ${\cal B}({\cal P}(\Omega))$.

To this purpose we consider the following measurable   mapping
$$
G : {\cal P}(\Omega) \to {\cal P}(E),\;\;\; G(P)= P\circ X_0^{-1},
\;\;\; P \in {\cal P}(\Omega),
$$
where $P\circ X_0^{-1}(A) = P (X_0 \in A)$, $A \in {\cal B}(E)$. By
\eqref{chec} we deduce that
$$
\begin{array}{c}
 (P^x)_{x \in E} = \Lambda_1 \cap \Lambda_2,\;\;\; \text{where}
\\
\Lambda_1 = \bigcap_{\eta \in \Gamma}\big \{ P \in {\cal P}
(\Omega) \; :\; \int_{\Omega} \eta (\omega)P(d \omega) =0 \big\},\;\;\; \Lambda_2 =
 G^{-1}(\{
\delta_x\}_{x \in E }).
\end{array}
$$
Note that
 for any    $\eta \in B_b(\Omega)$,
 the mapping: $P \mapsto
    \int_{\Omega} \eta (\omega)P(d \omega)$ is Borel on
    ${\cal P}(\Omega)$
  (this is easy to verify
  if  in addition  $\eta \in C_b(\Omega)$; the general case
  follows by a monotone class argument). It follows that $\Lambda_1 \in {\cal B}({\cal P}(\Omega))$.

On the other hand,
$D = \{  \delta_x\}_{x \in E } \in {\cal B}({\cal P}(E))$ (this
follows from Lemma 1.39 in \cite{Ka}) and so $\Lambda_2 \in {\cal B}({\cal P}(\Omega))$.
  The claim is proved.

\hh \textit{III Step.} Considering the restriction $G_0$ of $G$ to
$(P^x)_{x \in E}$ we find  that the measurable mapping $G_0: $
 $(P^x)_{x \in E} \to \{ \delta_x\}_{x \in E} $
is  one to one and onto.
By a   result of Kuratowski (see Theorem A.1.3 in
\cite{Ka}) the inverse function $G_0^{-1}: \{ \delta_x\}_{x \in E } \to (P^x)_{x \in E}$
is also measurable.  Finally to show that $x \mapsto P^x(A) =
\int_{\Omega}1_{A}(\omega)P^x(d\omega)$ is Borel on $E $, for any $A
\in {\cal B}(E)$,  we observe that the mapping $x \mapsto
\delta_x$ from $E$ into $\{ \delta_x\}_{x \in E }$ is a measurable
isomorphism.

 \hh \textit{IV Step.} We fix $\mu \in {\cal P}(E)$ and show that
there exists a unique martingale solution $P^{\mu}$ given by
\eqref{es}.

We have only to prove uniqueness since it is clear that $P^{\mu}$
 in \eqref{es} is a
 martingale solution for $(A, \mu)$. Let $\bar P$
 be a martingale solution for $(A, \mu)$. We prove that it coincides
  with $P^{\mu}$.
 Similarly to the first step in the proof of Theorem \ref{counta}, we
consider the regular conditional probability $Q^{\omega}$ for $\bar P$
 given $\sigma(X_0)$ (the $\sigma$-algebra generated by $X_0$).
  We see that  there exists a $\bar P$-null set $N
\in \sigma(X_0)$ such that
 for any $\omega \not \in N$, the probability measure $Q^{\omega}$
 solves the martingale problem for $(A, \delta_{\omega(0)})$
  $= (A, \delta_{X_0(\omega) })$.

By the uniqueness assumption we deduce that $Q^{\omega} =
 P^{X_0(\omega)}$, $\omega \not \in N$. Setting $\bar E = E^{\bar P}$ and using also the measurability
 property,  we finish with
 $$
\begin{array}{c}
\bar P (A)  = \bar E [ \bar E [1_A \setminus \sigma(X_0)] ] =
 \bar E [Q^{\omega}(A)] = \bar E [P^{X_0(\omega)}(A)]
 \\
= \int_{E} P^x(A) \mu (dx)= P^{\mu}(A),\;\;\; A \in {\cal B}(E).
\end{array}
$$
\end{proof}


\begin{remark} \label{strongm}{\em  Under the  assumptions of Theorem \ref{one} one can introduce the  semigroup
  $(P_t)$, $P_t : B_b(E) \to B_b(E)$,
  $P_t f(x) = \int_{C_E[0, \infty)} f(\omega(t))P^x (d\omega)$, for $f \in B_b(E)$, $t \ge 0$, $x \in E$.
   Combining  Theorem \ref{one} and Theorem 4.4.2  in \cite{EK} one proves the strong Markov property for a martingale solution $X$  for $(A, \mu)$. This means that, for any a.s. finite ${\cal F}_t^X$- stopping time $\tau$ one has:
   $ E [f(X_{t+ \tau}) \setminus {{\cal F}_{\tau}} ]$ $ =
    P_t f (X_{\tau}),$ $t \ge 0, $ $ f \in B_b(E).$  }
\end{remark}
  By the  previous theorems we get the following useful result.
\begin{corollary} \label{ria} Suppose that the operator $A$ satisfies Hypothesis \ref{bpt} and assume the following two conditions:

(i)  for any $x \in E$, there exists a (probability) martingale
solution
 $P^x$ for $(A, \delta_x)$;

(ii)  for any $x \in E$, any two (probability) martingale solutions $P^x_1$ and $P^x_2$ for $(A, \delta_x)$ have the same one dimensional marginal distributions (see \eqref{marg}).

Then the
   martingale
   problem for  $A$ is well-posed. In addition, $(P^x)$ depends measurably on $x$
and so   formula \eqref{es} holds for any $\mu \in {\cal P}(E)$.
\end{corollary}

\subsection {The localization principle}

\hh Let us first introduce the stopped martingale problem
following   Section 4.6 in \cite{EK}.

Let  $A$ be a linear operator,
 $A: D(A) \subset C_b(E) \to B_b(E)$. Consider $\mu  \in {\cal P}(E)$ and
an open set  $U \subset E$.

 An $E$-valued stochastic process $Z = (Z_t)_{t \ge 0}$
defined on some  probability space $(\Omega, {\cal F}, P)$
with continuous trajectories
is a {\it solution of the stopped martingale problem for $(A, \mu, U)$} if,
  the law of $Z_0$ is $\mu$ and the
 following conditions hold:

\hh (i) $Z_t = Z_{t \wedge \tau}$, $P$-a.s, where
\begin{align} \label{ta1}
\tau = \tau^Z_U= \inf \{ t \ge 0 \; : \; Z_t \not \in U  \}
\end{align}
($\tau = + \infty$ if the set is empty; it turns out that  this exit time $\tau$ is an
${\cal F}_t^Z$-stopping time);

\hh (ii) for any $f \in D(A)$,
\begin{equation}\label{mart1}
    M_{t\wedge \tau}(f) = f(Z_t) - \int_0^{t \wedge \tau}
    A f(Z_s) ds, \;\;\; t \ge 0,
\end{equation}
is a martingale with respect to the natural filtration $({\cal F}_t^Z)$.

\smallskip

{\sl  The  next key result
 shows
 that if the (global) martingale problem for $A$ is well-posed then also the
 stopped martingale problem for $(A, \mu, U)$ is well-posed for any choice
  of $(U, \mu)$. }

 A   related statement is given in Theorem 4.6.1 of \cite{EK}
which is based on Lemma 4.5.16.
 However such theorem
 requires   uniqueness for the (global) martingale problem   in the class of  all  c\`adl\`ag martingale solutions;
  actually, it is not clear  how to modify  the proof of Lemma 4.5.16
  in order to have the same statement of the lemma but in the case of   continuous martingale solutions.

 \begin{theorem} \label{key} Assume that $A$ verifies Hypothesis \ref{bpt} and that
 the  martingale problem for $A$ is well-posed.

  Then also the
 stopped martingale problem for $(A, \mu, U)$ is well-posed for any  $\mu \in {\cal P}(E)$ and for any open set $U$ of $E$.
 \end{theorem}

The proof is based on the following technical lemma which provides a kind of extension property for solutions to the stopped martingale problem (a  related result is   Lemma
 4.5.16 in \cite{EK} which is proved
 in the class of   c\`adl\`ag martingale solutions).

 We denote by $\tau_U: C_E [0, \infty)$ $\to [0, \infty]$  the exit time from U.


 \begin{lemma} \label{stop} Let $A$ be  a linear operator as in \eqref{lin}. Suppose that for any $x \in E$ there exists a (probability) martingale solution $P^x$ for $A$  and that $(P^x) $ depends measurably on $x$ (see \eqref{mes1}).
 Let $\mu \in {\cal P}(E)$ and $U$ be an open set of $E$. Let  $Z= (Z_t)$ be a
 martingale solution  for the stopped martingale problem  for $(A, \mu, U)$.

Then, for any $T>0$,
there exists a (probability) martingale  solution $ P_T$ for $(A,\mu)$ such that if $X$ is the canonical process on $ (C_E [0, \infty ), {\cal B}(C_E [0, \infty )), P_T) $ (see  \eqref{can})
 then $(X_{t \wedge \tau_U \wedge T})_{t \ge 0}$ and
$(Z_{t \wedge \tau_U^Z  \wedge T})_{t \ge 0}$
$= (Z_{t  \wedge T})_{t \ge 0}$ have the same law.
\end{lemma}
 \begin{proof}
 \textit{I Step. Construction of $P_T$.}

Our construction
 is  inspired by   page 271 of \cite{letta}. Let $Z$ be defined
on some probability space $(\Omega, {\cal F}, P)$ and introduce
\begin{equation} \label{tu6}
\tau = \tau_U^Z \wedge T.
\end{equation}
We consider the  measurable space  $\Omega_* = \Omega \times C_E[0, \infty)$ endowed with the product $\sigma$-algebra
  ${\cal F}_*$ $={\cal F} \otimes {\cal B } (C_E[0, \infty))$. On this product space, using the measurability of $x \mapsto P^x$, we consider  a probability  measure $P_*$ defined by
 the formula
 $$
\int_{\Omega_*} \! \! f(\omega, \omega') P_* (d\omega , d\omega') := \int_{\Omega} P(d\omega) \int_{C_E[0, \infty) } \! \! \! \! f(\omega, \omega') \, P^{Z_{\tau(\omega)}(\omega)}\, (d \omega'),
$$
for any real bounded and measurable function $f$ on  $\Omega \times C_E[0, \infty)$ (according to pages 19-20 in \cite{Ka},  $P^{Z_{\tau(\omega)}(\omega)}\, (d \omega')$ is a kernel from $\Omega$ into $C_E[0, \infty))$. Note that if $f(\omega, \omega')= f(\omega)$
then $E^{P_*}[f]= E^{P}[f]$ (here $E^P$ and $E^{P_*}$ denote  expectations on $(\Omega, {\cal F}, P)$ and
$(\Omega_*, {\cal F}_*, P_*)$ respectively).
 Then define
$$
J = \{ (\omega, \omega') \in \Omega_* \, :\, Z_{\tau(\omega)}(\omega) = \omega'(0) \}.
$$
Since $\omega \mapsto Z_{\tau(\omega)}(\omega)$ is ${\cal F}$-measurable, it is clear that $J \in {\cal F}_*$. Moreover we have $P_* (J)=1$ since $P^x(\omega' \, :\, \omega'(0)=x)=1$, $x \in E$. We restrict the events of $ {\cal F}_*$ to $J$ and consider the probability space $(J, {\cal F}_*, P_*)$.

Using that $\tau< \infty$,  we define a measurable mapping $\phi : J \to C_E[0, \infty)$ as follows
$$
\phi_t(\omega, \omega') = \begin{cases}
 Z_t(\omega),\;\; t\le \tau(\omega)
\\
 \omega' \big(t - \tau(\omega) \big),\;\; t> \tau(\omega)
\end{cases},\;\;\; \omega \in \Omega,\; \omega' \in C_E[0, \infty),\;\; t \ge 0
$$
(or $\phi_t(\omega, \omega')=  Z_t(\omega) 1_{\{ t \le \tau(\omega) \}} +  \omega' (t - \tau(\omega)) 1_{\{ t > \tau(\omega) \}} $, $t \ge 0$). Equivalently, $\phi= (\phi_t)$ is an $E$-valued continuous stochastic process. Note that $\tau_U^Z(\omega)= \tau_U^\phi (\omega, \omega')$, for any $(\omega, \omega') \in \Omega_*.$
 The required measure $P_T$ will be  the image probability distribution of $P_*$ under $\phi$, i.e.,
$$
P_T (B) = P_* (\phi^{-1}(B)),\;\;\; B \in   {\cal B}(C_E [0, \infty )).
$$
By the previous construction the  fact  that   $(X_{t \wedge \tau_U \wedge T})_{t \ge 0}$ and
  $ (Z_{t  \wedge T})_{t \ge 0}$ have the same law can be easily proved. Indeed, for any $ B \in
    {\cal B}(C_E [0, \infty )),$
$$
\begin{array}{c}
P_T( X_{ \cdot \wedge \tau_U \wedge T} \in B) = P_T( \omega' \in C_E[0, \infty) \; :\;    \omega'( \cdot \wedge \tau_U \wedge T) \in B)
\\
 =P_* (   \phi_{\cdot
    \wedge \tau_U^{\phi} \wedge T}  \in B ) =
    E^{P_*} [ 1_B( Z_{\cdot
    \wedge \tau_U^{Z} \wedge T} )] =
    P (  Z_{\cdot
    \wedge \tau_U^{Z} \wedge T}  \in B ).
\end{array}
$$
\textit{II Step. The measure  $P_T$ is a martingale solution for $(A, \mu)$.}

First we have
 $P_T (X_0 \in C)= P(Z_0 \in C)= \mu (C)$, for any $C \in {\cal B}(E)$.

Now we check the martingale property.  For fixed $0 \le t_1 < \ldots   < t_{n+1}$, $f \in D(A)$ and  $h_1, \ldots, h_n
\in C_b (E)$, we have to show that (using the canonical process $X$ defined in \eqref{can})
\begin{align}\label{equi1}
& E^{P_T} \big[ \big( M_{t_{n+1}}(f) - M_{t_{n}}(f) \big) \cdot \prod_{k=1}^n h_k(X_{t_k})
\big] =0,
\\ \nonumber  \text{where}\;\;
 & M_t(f)(\omega') := \omega'(t) - \int_0^t A f(\omega'(s)) ds, \;\; t \ge 0,\;\; \omega' \in C_E[0, \infty).
 \end{align}
Note that $\big( M_{t_{n+1}}(f) - M_{t_{n}}(f) \big) \cdot \prod_{k=1}^n h_k(X_{t_k}) = R_1 + R_2$, where
$R_i: C_E[0, \infty) \to \R$, $i=1,2$,
$$
\begin{array}{c}
R_1 = \big( M_{t_{n+1} \wedge (\tau_U \wedge T)}(f) - M_{t_{n} \wedge (\tau_U \wedge T)}(f) \big) \cdot \prod_{k=1}^n h_k(X_{t_k}),
\\
R_2 = \big( M_{t_{n+1} \vee (\tau_U \wedge T) }(f) - M_{t_{n} \vee (\tau_U \wedge T)}(f) \big) \cdot \prod_{k=1}^n h_k(X_{t_k}).
\end{array}
$$
As for  $R_1$ we  note  that if $t_n \ge \tau_U \wedge T$,  then $R_1=0$;  so with $\tau = \tau_U^Z \wedge T$ as in \eqref{tu6} we find
$$
\begin{array}{c}
E^{P_T}[R_1] = E^{P_*}[R_1 (\phi) \, 1_{\{ t_n < \tau \}}]
\\ = E^{P_*} \Big[   \Big( f(Z_{t_{n+1} \wedge \tau }) - f(Z_{t_{n} \wedge \tau }) -
\int_{t_{n} \wedge \tau }^{t_{n+1} \wedge \tau } Af(Z_r) dr  \Big) \cdot \prod_{k=1}^n
h_k(Z_{t_k \wedge \tau}) \, \cdot  1_{\{ t_n < \tau\}}\Big ].
\end{array}
$$
Since $\prod_{k=1}^n
h_k(Z_{t_k \wedge \tau}) \, \cdot  1_{\{ t_n < \tau\}}$ is bounded and ${\cal F}_{t_n}^Z$-measurable, using the martingale property \eqref{mart1} we find that $E^{P_T}[R_1]=0.$

Let us consider $R_2$ and note that $R_2 =0$ if $\tau_U \wedge T \ge t_{n+1}$. Set $C_E = C_E[0, \infty)$ and define
\begin{align*}
\Lambda (\omega, \omega') = f(\omega' (t_{n+1} \vee \tau(\omega) - \, \tau(\omega))) -  f(\omega' (t_{n} \vee \tau(\omega) - \, \tau(\omega)))\\ -
\int_{t_{n} \vee \tau (\omega)}^{t_{n+1} \vee \tau (\omega)} Af(\omega'(r- \tau(\omega))) dr,\;\;\; \omega \in \Omega, \; \omega' \in C_E.
\end{align*}
Since $(P^x)$ are martingale solutions, we have
\begin{equation} \label{sai}
\int_{C_E} \Lambda(\omega, \omega' ) F(\omega, \omega') P^x (d \omega') =0, \;\; \;\;\; \omega \in \Omega, \, x \in E,
\end{equation}
for any $F: \Omega \times C_E \to \R$, bounded and ${\cal F}_*$-measurable and such that $F (\omega, \cdot )$ is ${\cal F}_{t_n \vee \tau(\omega) - \tau(\omega)}^X-$ measurable, for any $\omega \in \Omega$.
Hence
\begin{align*}
& E^{P_T}[R_2] = E^{P_*}[R_2 (\phi) \, 1_{\{ t_{n+1} > \tau \}}]
= E^{P_*} \Big[    \Lambda
 \cdot \prod_{k=1}^n
h_k(\phi_{t_k }) \, \cdot  1_{\{ t_{n+1} > \tau\}}\Big ]
\\
& = \int_{\Omega} 1_{\{ t_{n+1} > \tau(\omega)\}} \cdot \! \! \! \! \prod_{t_k \le \tau(\omega)}
\! \! h_k(Z_{t_k} (\omega) )  P(d\omega) \! \!  \int_{C_E}  \! \! \Lambda(\omega, \omega')
  F(\omega, \omega') P^{Z_{\tau (\omega)}(\omega)}(d \omega')
\end{align*}
with $F(\omega, \omega') = \prod_{t_k > \tau (\omega) }
h_k(\omega'(t_k -\tau(\omega) ))$ and so by \eqref{sai} we get $E^{P_T}[R_2] =0$.
We have found that \eqref{equi1} holds and this completes the proof.
\end{proof}

\medskip
\begin{proof} [Proof of Theorem \ref{key}] \textit{Existence.}  Consider a martingale solution $X$ for $(A, \mu)$ and set $Z_t = X_{t \wedge \tau_U^X}$, $t \ge 0$. Note that $\tau_U^X = \tau_U^Z$. By the optional sampling theorem  we deduce that $Z= (Z_t)$ is a   solution of the
 stopped martingale problem for $(A, \mu, U)$.

\noindent \textit{Uniqueness.}   Since $A$
 satisfies Hypothesis \ref{bpt} we know by Theorem \ref{one} that the  martingale solutions $P^x$    depend measurably on $x$.

Let $Z^1$ and $Z^2$ be two solutions for the stopped martingale problem for $(A, \mu, U)$. To show that they have the same law it is enough to prove that, for any $T>0$, the processes $(Z_{t \wedge T}^1)$ and $(Z_{t \wedge T}^2)$  have the same law.

Fix $T>0$. By Lemma \ref{stop} there exist  martingale solutions $P^1$ and $P^2$ for $(A,\mu)$ such that if $X$ is the canonical process on $ (C_E [0, \infty ), {\cal B}( C_E [0, \infty )), P^k) $,  then  $(X_{t \wedge \tau_U \wedge T})_{t \ge 0}$ and
$(Z_{t   \wedge T}^k)_{t \ge 0}$, $k=1,2$, have the same law.
Since by hypotheses $P^1 = P^2$ we obtain easily the assertion.
\end{proof}

From Theorem \ref{key} we get

\begin{corollary} \label{loc}
Let $A_1$ and  $A_2 $ be linear operators with common domain $D(A_1) = D(A_2) = D \subset C_b(E) $ with values in $B_b(E)$. Suppose that  Hypothesis \ref{bpt} is satisfied. Let $U$ be an open subset of $E$ such that
\begin{equation}\label{loc1}
    A_1 f(x) = A_2 f(x),\;\;\; x \in U,\;\; f \in D.
\end{equation}
 If the martingale problem for $A_1$ is well-posed then the stopped martingale problem for $(A_2, \mu, U)$ is well-posed  for any $\mu \in {\cal P}(E)$.
\end{corollary}
\begin{proof} \textit{Existence.} If $X$ is a solution of the martingale problem for $(A_1, \mu)$ defined on $(\Omega, {\cal F}, P)$ then
$Z= (X_{t \wedge \tau})$ is a solution for the stopped martingale problem for
$(A_1, \mu, U)$, with $\tau = \tau_U^X$. Since, for any $f \in D$, $t \ge 0,$
$$
 f(X_{t \wedge \tau}) - \int_0^{t \wedge \tau}  A_1f(X_{s}) ds
  =  f(X_{t \wedge \tau}) - \int_0^{t \wedge \tau}  A_2 f(X_{s}) ds
$$
we see that $Z$ is also a solution for the stopped martingale problem for
$(A_2, \mu, U)$ (note that $X_0(\omega) \not \in U$ implies $\tau(\omega)=0$ and  $X_0(\omega)  \in U$ implies $\tau(\omega)>0$, $\omega \in \Omega$).

\hh \textit{Uniqueness.} Assume now that $Z$ and $W$ are both solutions for the stopped martingale problem for
$(A_2, \mu, U)$. It follows that they are also solutions for the stopped martingale problem for $(A_1, \mu, U)$. By Theorem \ref{key} we deduce
that $Z$ and $W$ have the same law.
\end{proof}


\smallskip
The following result is a kind of  converse of Theorem \ref{key} and  gives conditions under which
  uniqueness for  stopped martingale problems implies uniqueness for the global martingale problem. It is a modification of Theorem 4.6.2  in \cite{EK}.


\begin{theorem} \label{uni}  Assume that $A$ verifies Hypothesis \ref{bpt} and
    that for any $x \in
  E$ there exists a martingale solution  for  $(A, \delta_x)$.

Suppose that
 there exists
a sequence of open sets $U_k \subset E$ with $\cup_{k \ge 1} U_k = E$  such that for any  $\mu \in {\cal P}(E)$, for any  $k \ge 1 $, we have uniqueness  for   the
 stopped martingale problem for $(A, \mu, U_k)$.

Then the martingale problem for $A$ is well-posed.
\end{theorem}
\begin{proof} By Corollary \ref{ria} we  have to prove that for a fixed $x \in E$ any  two martingale solutions
 $P^1$ and $P^2$ for $(A, \delta_x)$ have the same one dimensional marginal distribution.
Thus using  the canonical process $(X_t)$ given in \eqref{can} and
 a uniqueness result for   the Laplace transform, it is enough to  show
 that, for any
$\lambda >0$, $f \in C_b(E),$
\begin{equation}\label{lap17}
    E^{1} \Big[ \int_0^{+\infty} e^{- \lambda t} f(X_t) dt\Big] = E^{2} \Big[ \int_0^{+\infty} e^{- \lambda t} f(X_t) dt\Big],
\end{equation}
with $E^j = E^{P^j}$, $j=1,2$. We first introduce ${\cal S}= \{ U_k^{(j)} \}_{k \ge 1, \, j \ge 1}$, where $U_k^{(j)} = U_k$, $k,\, j \ge 1$. Then we  enumerate ${\cal S}$ using positive integers and find
 ${\cal S} = (V_i)_{i \ge 1}$ (so each $U_k$ appears infinitely many times in $(V_i)_{i \ge 1}$).

{\sl To prove \eqref{lap17} we show that for any $\lambda >0$ there exist $\mu_i \in {\cal P}(E)$, $i \ge 1$, such that, for any (probability)  martingale solution
 $P$  for $(A, \delta_x)$,  we have that
 $$
g(\lambda, f)  :=E^P \Big[ \int_0^{+\infty} e^{- \lambda t} f(X_t) dt\Big]
 $$
 can be computed, for any  $f \in C_b(E)$, using the (unique) laws of  solutions of the stopped
 martingale problems for $(A, \mu_i, V_i)$, $i \ge 1$.
}

The previous claim can be   proved adapting  the proof of
 Theorem 4.6.2  in \cite{EK};   we  give a sketch of proof for the sake of completeness.

Define, for any $\omega \in C_E [0, \infty)= C_E$,  $\tau_0(\omega) =0$ and, for $i \ge 1$,
$$
 \tau_i(\omega) =  \inf \{ t \ge \tau_{i-1} (\omega)\; :\;
 \omega (t) \not \in V_i \}
$$
 (where inf $\emptyset = \infty$). By Proposition 2.1.5 in \cite{EK} each $\tau_i$ is an ${\cal F}_t^X$-stopping time.   Moreover, for any $\omega \in C_E $, $\tau_i (\omega) \to +\infty$, as $i \to \infty$.

 Indeed
 let $\tau = \sup_{i} \tau_i$ and suppose that for some $\omega \in C_E $ we have $\tau(\omega) < +\infty$. Then there exists $U_{k(\omega)} $ such that $\omega (\tau(\omega)) \in U_{k(\omega)}$. It follows that for $s \in [0, \tau(\omega)[$ close enough to $\tau(\omega)$ we have  $\omega (s) \in U_{k(\omega)}$.  Then we can find an integer $i= i(\omega)$ large enough such that $\omega (\tau_i(\omega)) \in U_{k(\omega)}$ and also
  $V_{i(\omega)} = U_{k(\omega)}$; this is a contradiction since by construction $\omega (\tau_i(\omega)) \not \in V_{i(\omega)}$.

Let $P$ be  any   martingale solution
   for $(A, \delta_x)$ on $(C_E, {\cal B}(C_E))$ and fix $\lambda>0.$
 We find, setting $E = E^P$,
\begin{align} \label{f5}
\begin{array}{c}
g(\lambda, f) = \sum_{i \ge 1}  E \Big[ 1_{\{ \tau_{i-1}< \infty\}}\,  \int_{\tau_{i-1}}^{\tau_i} e^{- \lambda t} f(X_t) dt\Big]\\
 \sum_{i \ge 1}  E \Big[e^{- \lambda \, \tau_{i-1}} \; 1_{\{ \tau_{i-1}< \infty\}} \int_{0}^{\eta_i} e^{- \lambda t} f(X_{t\wedge \eta_i \, + \, \tau_{i-1}}) dt\Big],
\end{array}
\end{align}
where on $\{ \tau_{i-1} <\infty \}$, we define $\eta_i := \tau_{i} - \tau_{i-1} $   so that  $\eta_i = \inf \{ t \ge 0 \; :\;
  X_{t + \, \tau_{i-1}} \not \in V_i \}$.
 For any $i \ge 1$ such that $P (\tau_{i-1} < \infty)>0$ define $\mu_i \in {\cal P}(E),$
\[
\mu_i (B)= \frac{E \big[e^{- \lambda \, \tau_{i-1}} \; 1_{\{ \tau_{i-1}< \infty\}} \, 1_B (X_{\tau_{i-1}}) \big]}
{E \big[e^{- \lambda \, \tau_{i-1}} \; 1_{\{ \tau_{i-1}< \infty\}}  \big]},\;\; B \in {\cal B}(E),
\]
and the stochastic process $Y^i = (Y^i_t)$, $Y^i_t := X_{t\wedge \eta_i \, + \, \tau_{i-1}}$, $t \ge 0$, defined on  $(C_E, {\cal B}(C_E), P_i)$ where $
P_i (C)$ $ = \frac{E \big[ e^{- \lambda \, \tau_{i-1}} \; 1_{\{ \tau_{i-1}< \infty\}} \, 1_C]} { E \big[e^{- \lambda \, \tau_{i-1}} \; 1_{\{ \tau_{i-1}< \infty\}} ]}$, $C \in {\cal B}(C_E)$.
It follows that $\mu_1 = \delta_x$.
We need to show that $Y^i$ is a solution of the stopped martingale problem  for $(A, \mu_i, V_i)$.
Note that
\begin{equation} \label{unic}
\tau_{V_i}^{Y^i} = \eta_i,\;\;\; i \ge 1.
\end{equation}
It is also clear that the law of $Y_0^i$ is $\mu_i$ and also that $Y_t = Y_{t \wedge \eta_i}$, $t \ge 0$.
It remains to check the martingale property \eqref{mart1}.
 To this purpose it is enough to  prove that $\tilde X = (X_{t \, + \, \tau_{i-1}})_{t \ge 0}$ defined on $(C_E, {\cal B}(C_E), P_i)$ is a (global) martingale solution for $(A, \mu_i)$.

We fix $t_2 > t_1 \ge 0$ and consider $G \in {\cal F}_{\tau_{i-1} + t_1} = {\cal F}_{t_1}^{\tilde X} $. For any $T>0$ we have
 with $\alpha_{i-1} = E \big[e^{- \lambda \, \tau_{i-1}} \; 1_{\{ \tau_{i-1}< \infty\}} ]$,
\begin{align*}
E^{P_i} \Big [ \Big ( f(\tilde X_{t_2 \wedge T}) -   f(\tilde X_{t_1 \wedge T }) - \int_{t_1 \wedge T}^{t_2\wedge T } A f(\tilde X_s)ds \Big) \, 1_{G} \Big]
\\
=  \frac{1}{\alpha_{i-1}} E \Big [e^{- \lambda \, \tau_{i-1}} \; 1_{\{ \tau_{i-1}< \infty\}}  \Big ( f( X_{ (t_2 + \tau_{i-1}) \wedge T}) -   f( X_{(t_1 + \tau_{i-1})  \wedge T }) \\ - \int_{(t_1 + \tau_{i-1})  \wedge T}^{(t_2 + \tau_{i-1}) \wedge T } A f(X_s)ds \Big) \, 1_{G} \Big]\\
= \frac{1}{\alpha_{i-1}} E [ \big ( M_{(t_2 + \tau_{i-1}) \wedge T }(f) - M_{(t_1 + \tau_{i-1}) \wedge T }(f) \big) \, Z_1 ] =0,
\end{align*}
where $Z_1 :=  1_G \,  e^{- \lambda \, \tau_{i-1}} \; 1_{\{ \tau_{i-1}< \infty\}} $ is bounded and ${\cal F}_{\tau_{i-1} + t_1}$-measurable. Note that the last quantity is zero by the optional sampling theorem
 (see also Remark 2.2.14 in \cite{EK}).
Now we pass to the limit as $T \to \infty$ and get
 $E^{P_i} \Big [ \Big ( f(\tilde X_{t_2 }) -   f(\tilde X_{t_1 }) - \int_{t_1 }^{t_2 } A f(\tilde X_s)ds \Big) \, 1_{G} \Big] =0$. To justify such limit procedure one can use the  estimate
$$
\begin{array}{c}
 e^{- \lambda \, \tau_{i-1}} \; 1_{\{ \tau_{i-1}< \infty\}}
\int_{0}^{(t_2 + \tau_{i-1} ) \wedge T } |A f(X_s)| ds
 \le  Z_0,  \;\;\; T>0,
\end{array}
$$
where $ Z_0 :=\|Af \|_{\infty} (t_2 +  \tau_{i-1})
e^{- \lambda \, \tau_{i-1}} \; 1_{\{ \tau_{i-1}< \infty\}}$ is bounded.

Let us denote  by $Q_i$ the law of $Y^i$ on $(C_E, {\cal B}(C_E))$. We have (using \eqref{unic})
\begin{equation} \label{d4}
g(\lambda, f) =
 \sum_{i \ge 1} \alpha_{i-1} \,  E^{P_i} \Big[ \int_{0}^{\eta_i} e^{- \lambda t} f(Y^i_t) dt\Big]
=  \sum_{i \ge 1} \alpha_{i-1} \,  E^{Q_i} \Big[ \int_{0}^{\tau_{V_i}^{X}} e^{- \lambda t} f(X_t) dt\Big].
\end{equation}
Note that, for any $B \in {\cal B}(E)$,
\begin{equation} \label{ft9}
\begin{array}{c}
\mu_{i+1} (B) =  \frac{1}{\alpha_{i}} E^{P} \Big [e^{- \lambda \, \tau_{i-1}}
 \; 1_{\{ \tau_{i-1}< \infty\}}
  e^{- \lambda \, \eta_i}
 \; 1_{\{ \eta_i < \infty\}} 1_B (X_{\tau_i}) \Big]
 \\
= \frac{\alpha_{i-1}}{\alpha_{i}} E^{P_i} \Big [
  e^{- \lambda \, \eta_i}
 \; 1_{\{ \eta_i < \infty\}} 1_B (Y_{\eta_i}) \Big]
 =  \frac{\alpha_{i-1}}{\alpha_{i}} E^{Q_i} \Big [
  e^{- \lambda \, \tau_{V_i}^X}
 \; 1_{\{ \tau_{V_i}^X < \infty\}} 1_B (X_{\tau_{V_i}^X}) \Big],
 \end{array}
\end{equation}
and, for $i \ge 1,$
$$
 \alpha_{i} = \alpha_{i-1} E^{P_i} \big[e^{- \lambda \, \eta_{i}} \; 1_{\{ \eta_{i}< \infty\}} ]
  = \alpha_{i-1}
  E^{Q_i} \big [
  e^{- \lambda \, \tau_{V_i}^X}
 \; 1_{\{ \tau_{V_i}^X < \infty\}}  \big]
= \prod_{k=1}^i E^{Q_k} \big [
  e^{- \lambda \, \tau_{V_k}^X}
 \; 1_{\{ \tau_{V_k}^X < \infty\}}  \big].
$$
 Now $\mu_1 = \delta_x$ determines $Q_1$ by uniqueness of
 the stopped martingale problem and then $Q_1$ determine $\mu_2$ by \eqref{ft9}. Proceeding in this way,
 $Q_1, \ldots, Q_i$ determine $\mu_{i+1}$ and again by uniqueness
  this characterize $Q_{i+1}$, $i \ge 1$. By \eqref{d4}, for any $\lambda>0$, for any $f \in C_b(E)$,
 $g(\lambda, f)$ is completely determined independently of the martingale solution $P$ for $(A, \delta_x)$ we have chosen. This completes the proof.
 \end{proof}

\smallskip
Combining  Theorems \ref{key} and \ref{uni}  and using Corollary \ref{loc} we get the following {\it localization principle}. It extends  Theorem 6.6.1 in \cite{SV79}   and shows that to perform the localization procedure it is enough to have existence of (global) martingale solutions of any $x \in E$.

\begin{theorem} \label{uni1}  Assume that $A$ verifies Hypothesis \ref{bpt} and
    that for any $x \in
  E$ there exists a martingale solution  for  $(A, \delta_x)$.
  Suppose that
 there exists
a family $\{ U_j\}_{j \in J}$ of open sets $U_j \subset E$  with $\cup_{j \in J} U_j = E$  and  linear operators $A_j$ with the same domain of $A$, i.e., $ A_j: D(A) \subset C_b (E) \to B_b(E)$, $j \in J $ such that

\hh
i)  for any $j \in J$, the martingale problem for $A_j$ is well-posed.

\hh ii) for any $j \in J$, $f \in D(A)$, we have
$
 A_j  f(x) = A f(x),\;\; x \in U_j.
$

Then the martingale problem for $A$ is well-posed. In addition, $(P^x)$ depends measurably on $x$
and so   formula \eqref{es} holds for any $\mu \in {\cal P}(E)$.
  \end{theorem}
\begin{proof} Since $E$ is   a  separable metric space we can consider a countable sub-covering of $\{ U_j\}_{j \in J}$ that we denote by $(U_k)_{k \ge 1}$ (i.e., $(U_k)_{k \ge 1} \subset \{ U_j\}_{j \in J}$ and $\cup_{k \ge 1} U_k =E$).

By Corollary \ref{loc}   we deduce that
 the
 stopped martingale problem for $(A, \mu, U_k)$ is well-posed for any  $\mu \in {\cal P}(E)$ and for any open set $U_k$. Applying Theorem \ref{uni} we obtain the first assertion.
 The measurability assertion follows from Corollary \ref{ria}.
 \end{proof}

\smallskip
 We finally
     mention another result on well-posedness in which one considers an increasing sequence of open sets.
It extends  Corollary 10.1.2 in \cite{SV79}
 and   can be proved  as   Theorem 6.6.3 in \cite{EK}.

\begin{theorem} \label{exi} Let $\mu \in {\cal P}(E) $ and  let $(U_k)_{k \ge 1}$ be an  increasing sequence  of open sets in $E$, i.e., $U_k \subset U_{k+1},$ $k \ge 1.$
 Suppose that, for any $k \ge 1$, there exists a unique (in law) solution  for the stopped martingale problem for $(A, \mu, U_k)$.

Let $Z^k$ be a solution for the stopped martingale problem for $(A, \mu, U_k)$ defined on a probability space
 $(\Omega^k,{\cal F}^k, P^k )$ and consider
$
\tau_k = \tau_k^{Z^k} = \inf \{ t \ge 0 \; :\; Z_t^k \not \in U_k \}.
$

There exists a unique solution for the martingale
problem for $(A, \mu)$ if, for any $t>0,$
\begin{equation}\label{stop2}
    \lim_{k \to \infty} P^k (\tau_k \le t) =0.
\end{equation}
\end{theorem}

\begin{proof} One can adapt without difficulties the  proof of Theorem 6.6.3 in \cite{EK} which deals with c\`adl\`ag martingale solutions. To this purpose,
using  \eqref{stop2}, one first proves that there exists a continuous process $Z_{\infty}$ with values in $E$ such that  the law of $Z^k$ converges in the  Prokhorov distance to the law of $Z_{\infty}$.
One checks that $Z_{\infty}$ is a solution of the martingale problem for $(A, \mu)$. Also the  uniqueness part can be proved as in \cite{EK}.
\end{proof}


\begin{corollary} \label{exi2}  Assume that $A$ verifies Hypothesis \ref{bpt}.
   Suppose that there exists an  increasing sequence  of open sets $(U_k)_{k \ge 1}$ in $E$  and  linear operators $A_k$ with the same domain of $A$. Moreover, assume:

\noindent
i)  for any $k \ge 1 $,   the martingale problem for $A_k$ is well-posed;

\noindent ii) for any $k \ge 1$, $f \in D(A)$, we have
$
 A_k  f(x) = A f(x),\;\; x \in U_k.
$

\noindent  For $x \in E$,   let $X^k= X^{k,x}$ be a martingale solution
for  $(A_k, \delta_x)$ defined on a probability space $(\Omega^k,{\cal F}^k, P^k )$;    define
$
\tau_k =\tau_k^x  = \inf \{ t \ge 0 \; :\; X_t^k \not \in U_k \}.
$

Then the martingale problem for $A$ is well-posed if, for any $x \in E$,   for any $t>0,$
\begin{equation}\label{stop21}
    \lim_{k \to \infty} P^k (\tau_k \le t) =0.
\end{equation}
  \end{corollary}
\begin{proof} By Theorem \ref{one} it is enough to prove that for any $x \in E$, the martingale problem for $(A, \delta_x)$ is well-posed.
 Let us fix $x \in E$. By Corollary \ref{loc} the stopped martingale problems for $(A, \delta_x, U_k)$ are well-posed, $k \ge 1$.
  If $X^k$ is a solution of the martingale problem for $(A_k, \delta_x)$ defined on $(\Omega^k, {\cal F}^k, P^k)$ then
$Z^k:= (X_{t \wedge \tau_k^x}^k)_{t \ge 0}$ is a solution for the stopped martingale problem for
$(A_k, \delta_x, U_k)$, with $\tau_k = \tau_{U_k}^{Z^k}$.
 If follows that \eqref{stop21} is just  \eqref{stop2}.
 By Theorem \ref{exi} there exists a unique martingale solution for $(A, \delta_x)$ and this finishes the proof.
\end{proof}

%

\end{document}